\documentclass[reqno,twoside]{amsart}

\usepackage{amsfonts,amsmath,amssymb,amsthm}
\usepackage{mathtools}
\usepackage{etoolbox}
\usepackage{color}
\usepackage[english]{babel}
\usepackage{hyperref}
\usepackage{microtype}
\usepackage{fullpage}
\usepackage{floatrow}
\usepackage[utf8]{inputenc}
\usepackage[T1]{fontenc}
\usepackage[mathscr]{euscript}

\usepackage{float}

\usepackage{placeins} 
\usepackage{flafter} 

\DisableLigatures{encoding = *, family = * }

\newtheorem{theorem}{Theorem}[section]

\newtheorem{lemma}[theorem]{Lemma}
\newtheorem{proposition}[theorem]{Proposition}

\numberwithin{equation}{section}

\def\NN{\mathbb{N}}
\def\RR{\mathbb{R}}
\def\CC{\mathbb{C}}
\def\PP{\mathcal{P}}
\def\B{\mathcal{B}}
\def\C{\mathcal{C}}
\def\M{\mathcal{M}}
\newcommand{\norm}[2]{{\left\|#1\right\|}_{#2}}
\newcommand{\fl}[2]{(-\Delta)^{\,#1}#2}
\newcommand{\ffl}[2]{(-d_x^{\,2})^{#1}#2}
\newcommand{\cns}{c_{N,s}}

\title[Control non-local Schr\"odinger equation]{Internal control for a non-local Schr\"odinger equation involving the fractional Laplace operator}

\author[U. Biccari]{}
\email{umberto.biccari@deusto.es}
\email{u.biccari@gmail.com}

\author{Umberto Biccari\textsuperscript{\,$\ast$}}  
\address{\textsuperscript{$\ast$}\, [1] Chair of Computational Mathematics, Fundaci\'on Deusto, Avenida de las Universidades 24, 48007 Bilbao, Basque Country, Spain} 
\address{[2]\,Facultad de Ingenier\'ia, Universidad de Deusto, Avenida de las Universidades 24, 48007 Bilbao, Basque Country, Spain.}
\email{umberto.biccari@deusto.es, u.biccari@gmail.com}

\thanks{This project has received funding from the European Research Council (ERC) under the European Union’s Horizon 2020 research and innovation programme (grant agreement NO: 694126-DyCon). This work was partially supported by the Grant MTM2017-92996-C2-1-R COSNET of MINECO (Spain), by the Elkartek grant KK-2020/00091 CONVADP of the Basque government and by the Grant FA9550-18-1-0242 of AFOSR}


\begin{document}
	
\maketitle

\begin{abstract}
We analyze the interior controllability problem for a non-local Schr\"odinger equation involving the fractional Laplace operator $\fl{s}{}$, $s\in(0,1)$, on a bounded $C^{1,1}$ domain $\Omega\subset\RR^N$. We first consider the problem in one space dimension and we employ spectral techniques to prove that, for $s\in[1/2,1)$, null-controllability is achieved by employing an $L^2(\omega\times(0,T))$ function acting in a subset $\omega\subset\Omega$ of the domain. This result is then extended to the multi-dimensional case by applying the classical multiplier method, joint with a Pohozaev-type identity for the fractional Laplacian. 
\end{abstract}

\section{Introduction and main results}\label{intro_sec}

Let $T>0$ be a real number, $\Omega\subset\RR^N$ be a bounded $C^{1,1}$ domain, $\Omega^c:=\RR^N\setminus\Omega$, and set $Q:=\Omega\times(0,T)$ and $Q^c:=\Omega^c\times(0,T)$. In this work, we analyze the controllability properties of the following non-local Schr\"odinger equation 
\begin{align}\label{fr_schr_control}
	\begin{cases}
		iu_t+\fl{s}{u} = h\chi_\omega, & (x,t)\in Q
		\\ 
		u\equiv 0, & (x,t)\in Q^c
		\\ 
		u(x,0) = u_0(x), & x\in\Omega.
	\end{cases}
\end{align}

In \eqref{fr_schr_control}, $u_0\in L^2(\Omega)$ is a given initial datum and the control $h\in L^2(\omega\times(0,T))$ is applied on a neighborhood $\omega$ of the boundary of $\Omega$, defined as 
\begin{align}\label{neigh_bd_def}
	\omega:=\mathcal{O}_{\varepsilon}\cap\Omega,\quad \mathcal{O}_{\varepsilon}:=\bigcup_{x\in\Gamma_0} B(x,\varepsilon), \quad\varepsilon>0,
\end{align} 
where
\begin{align}\label{partition}
	\Gamma_0=\big\{x\in\partial\Omega\,|\,(x\cdot\nu)>0\,\big\},\quad \Gamma_1=\partial\Omega\setminus\Gamma_0,
\end{align}  
and $\nu$ is the unit normal vector to $\partial\Omega$ at $x$ pointing towards the exterior of $\Omega$ (see Figure \ref{partition_fig}).

\begin{figure}
	\centering
	\includegraphics[scale=1]{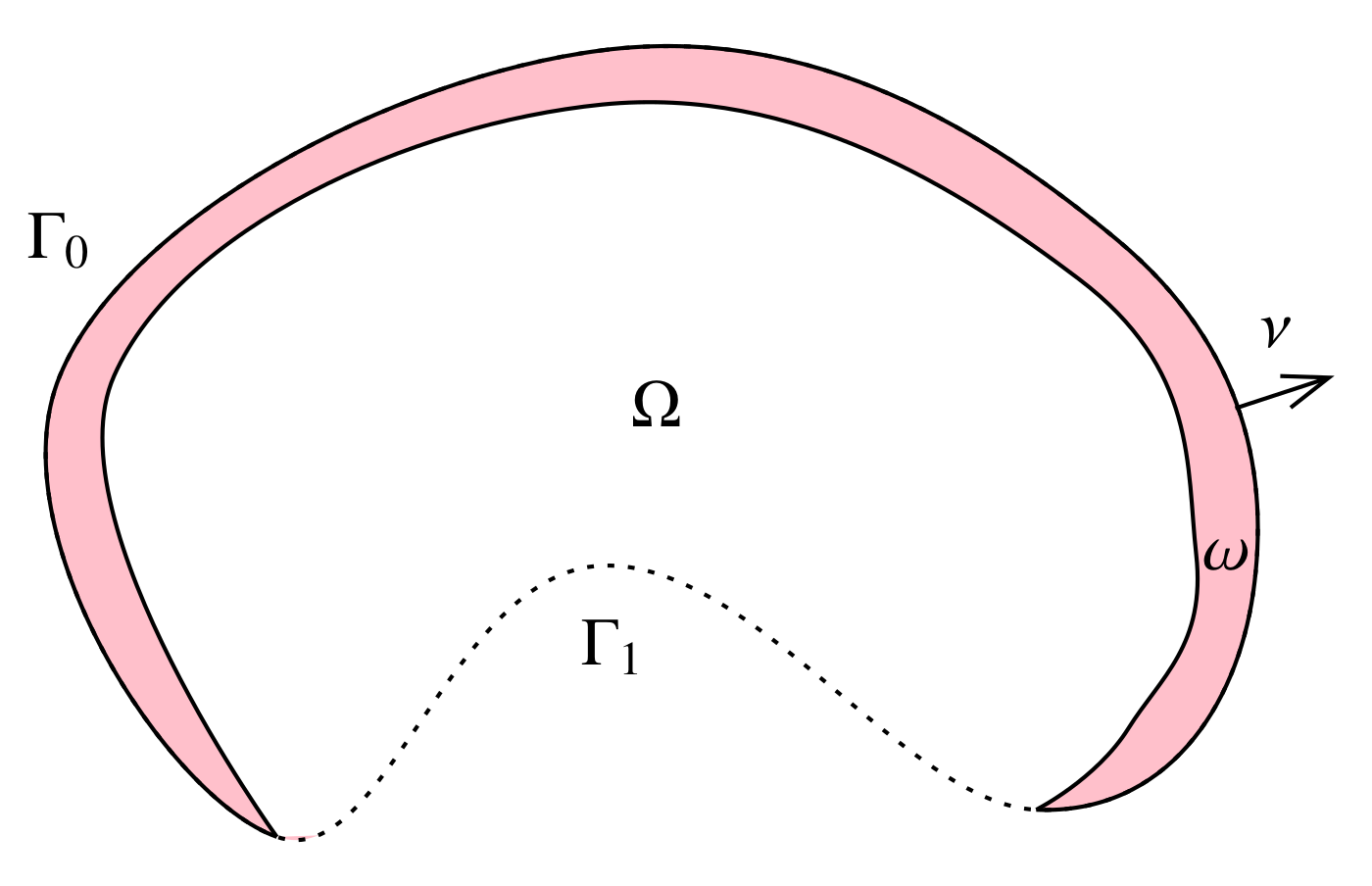}
	\caption{The domain $\Omega$ with the partition $(\Gamma_0,\Gamma_1)$ of its boundary and the neighborhood $\omega$ of $\Gamma_0$.}\label{partition_fig}
\end{figure}
	
Moreover, with $\fl{s}{}$ we denote the fractional Laplace operator whose precise definition will be given in the next section.

Fractional order operators (in particular the fractional Laplacian) have recently emerged as a modeling alternative in various branches of science. From the long list of phenomena which are more appropriately modeled by fractional differential equations, we mention: anomalous transport and diffusion (\cite{bologna2000anomalous,meerschaert2012fractional}), elasticity (\cite{dipierro2015dislocation}), image processing (\cite{antil2019external,gilboa2008nonlocal}), porous media flow (\cite{vazquez2012nonlinear}), and population dynamics (\cite{viswanathan1996levy}). In particular, space-fractional Schr\"odinger equations have been introduced by Laskin in quantum mechanics \cite{laskin2000fractional,laskin2002fractional,laskin2000quantum}), since they provide a natural extension of the standard local model when the Brownian trajectories in Feynman path integrals are replaced by L\'evy flights. Applications of this model may be found in the study of a condensed-matter realization of L\'evy crystals (\cite{stickler2013potential}). More recently, the fractional Schr\"odinger equation was introduced into optics by Longhi in \cite{longhi2015fractional}, with applications to laser implementation. In addition to that, concentration phenomena for the fractional Schr\"odinger equation have been studied in \cite{davila2015concentration,davila2014concentrating}.

Moreover, a number of stochastic models for explaining anomalous diffusion have been introduced in the literature. Among them, we quote the fractional Brownian motion, the continuous time random walk, the Lévy flights, the Schneider grey Brownian motion, and, more generally, random walk models based on evolution equations of single and distributed fractional order in space (see e.g. \cite{dubkov2008levy,gorenflo2007continuous,mandelbrot1968fractional}). In general, a fractional diffusion operator corresponds to a diverging jump length variance in the random walk. In the literature, the fractional Laplace operator is known as the generator of the so called s-stable Lévy process. We refer to \cite{gal2017nonlocal,valdinoci2009from} and the references therein for further details.

The controllability properties of fractional PDE are still not fully understood by the mathematical community, and only few results are currently available in the literature. 

In \cite{biccari2018controllability}, the null controllability of the one-dimensional fractional heat equation has been treated by using the gap condition on the eigenvalues. This result has been extended to the case of controls acting from the exterior of the domain in \cite{warma2018null}. Moreover, the controllability properties of the one-dimensional fractional heat equation under positivity constraints on the control have been studied in \cite{antil2019controllability,biccari2019controllability}. In space dimension $N\geq 2$, the best possible controllability result available for the fractional heat equation is the approximate controllability recently obtained in \cite{warma2019approximate}. 

Concerning wave-like models, the approximate controllability from the exterior of fractional wave equations has been proved in \cite{louis2018approximate,warma2018analysis}, while \cite{biccari2019null} treats the null-controllability of a one-dimensional fractional wave equation with memory. Nonetheless, as far as the author knows, there are currently no controllability results for the fractional Schr\"odinger equation.

In the present paper, we are interested in studying precisely this issue. In more detail, we are going to prove the following result.
\begin{theorem}\label{schr_control_thm}
Let $\Omega\subset\RR^N$ be a bounded $C^{1,1}$ domain and $s\in[1/2,1)$. Moreover, let $\omega\subset\Omega$ be a neighborhood of $\Gamma_0$, defined as in \eqref{neigh_bd_def}. 
\begin{enumerate}
	\item[(i)] If $s\in(1/2,1)$, for any $T>0$ and any given initial datum $u_0\in L^2(\Omega)$ there exists a function $h\in L^2(\omega\times(0,T))$ such that the corresponding solution $u$ of \eqref{fr_schr_control} satisfies $u(x,T)=0$.
	\item[(ii)] If $s=1/2$, there exists a strictly positive time $T_0>0$ such that the same controllability result as in (i) holds for any $T\geq T_0$. 
\end{enumerate} 
Besides, in both cases there exists a positive constant $\C>0$ such that 
\begin{align*}
	\norm{h}{ L^2(\omega\times(0,T))}\leq \mathcal C\,\norm{u_0}{L^2(\Omega)}.
\end{align*}
\end{theorem}

As it will be clear in the further sections, the requirement of strict positivity of the controllability time when $s=1/2$ arises naturally in the proof of our Theorem \eqref{schr_control_thm}. More details will be given later. We anticipate that, as we will see in Section \ref{1d_sec}, in the one-dimensional case we can provide an explicit lower estimate for $T_0$.

The proof of Theorem \ref{schr_control_thm} is based on the classical Hilbert Uniqueness Method (HUM, \cite{coron2009control,lions1988controlabilite,lions1988exact}), according to which the controllability of \eqref{fr_schr_control} is equivalent to the following observability inequality 
\begin{align}\label{schr obs}
	\norm{v_0}{ L^2(\Omega)}^2\leq \mathcal C\int_0^T\norm{v(t)}{ L^2(\omega)}^2\,dt,
\end{align}
where, for any $v_0\in L^2(\Omega)$, $v$ is the solution of the adjoint system 
\begin{align}\label{fr_schr_adj}
	\begin{cases}
		iv_t+\fl{s}{v}=0, & (x,t)\in Q
		\\ 
		v\equiv 0, & (x,t)\in Q^c
		\\ 
		v(x,0)=v_0(x), & x\in\Omega.
	\end{cases}
\end{align}

For obtaining this inequality, we will distinguish two cases in which we will use different methodologies: 
\begin{itemize}
	\item In one space dimension, \eqref{schr obs} will be proved by employing spectral techniques and Ingham-type estimates (\cite{ingham1936some}), following the by now standard approach described, for instance, in \cite{micu2004introduction}. In this framework, we will also show that $s\in[1/2,1)$ is a necessary condition for the positivity of the controllability result. Indeed, our analysis yields that, at least in one space-dimension, when the power of the fractional Laplacian is below the critical value $s=1/2$, equation \eqref{fr_schr_control} fails to be controllable. This fact is a direct consequence of the spectral results contained in \cite{kulczycki2010spectral,kwasnicki2012eigenvalues}. For completeness, we shall stress that, as we will see in Section \ref{1d_sec}, for the employment of the mentioned spectral techniques the requirement that $\omega$ is a neighborhood of the boundary is actually not necessary. Therefore, in the one-dimensional case the support of the control may be any general open subset.
	\item In dimension $N\geq 2$, \eqref{schr obs} will be obtained by applying the multiplier method (\cite{komornik1994exact}) and a Pohozaev-type identity for the fractional Laplacian, which has been proved in \cite{ros2014pohozaev} and extends to the fractional case the well-known identity due to Pohozaev (see \cite{pohozaev1965eigenfunctions}). 	
\end{itemize}

The rest of the paper is organized as follows. Section \ref{fl_sec} is devoted to the presentation of the functional setting in which we will work. Moreover, we recall there some results related to the fractional Laplace operator, the regularity of the associated Dirichlet problem, the Pohozaev-type identity and the principal spectral properties in the one-dimensional case. We conclude the section by discussing the well-posedness of the fractional Schr\"odinger equation \eqref{fr_schr_control}. In Section \ref{1d_sec}, we treat the one-dimensional controllability problem by employing standard spectral techniques. Section \ref{fl_schr_sec} will be devoted to the multi-dimensional case. We start by obtaining a Pohozaev-type identity for \eqref{fr_schr_control}, which we will later apply for proving the observability inequality \eqref{schr obs}. Our main result, Theorem \ref{schr_control_thm}, will then be a consequence of this inequality. Finally, Section \ref{open_pb_sec} is devoted to some open problems and perspectives related to our work. 

\section{Fractional Laplace operator: definition, Dirichlet problem and Pohozaev-type identity}\label{fl_sec}

In this Section, we introduce some preliminary results that will be useful in the remainder of the paper. 

We start by giving a rigorous definition of the fractional Laplace operator. To this end, for $s\in(0,1)$, let us consider the space
\begin{align*}
	\mathcal L^1_s\left(\RR^N\right) :=\left\{ u:\RR^N\longrightarrow\RR^N\,:\; u\textrm{ measurable },\;\int_{\RR^N}\frac{|u(x)|}{(1+|x|)^{N+2s}}\,dx<+\infty\right\}.
\end{align*}
For any $u\in\mathcal L_s^1\left(\RR^N\right)$ and $\varepsilon>0$, we set 
\begin{align*}
	(-\Delta)^s_{\varepsilon}\, u(x) = \cns\,\int_{|x-y|>\varepsilon}\frac{u(x)-u(y)}{|x-y|^{N+2s}}\,dy,\;\;\; x\in\RR^N,
\end{align*}
with $\cns$ an explicit normalization constant given by 
\begin{align*}
	\cns:=\frac{s2^{2s}\Gamma\left(\frac{N+2s}{2}\right)}{\pi^{N/2}\Gamma(1-s)},
\end{align*}
$\Gamma$ being the Euler Gamma function. The fractional Laplacian is then defined by the following singular integral
\begin{align}\label{fl}
	\fl{s}{u}(x) = \cns\,P.V.\,\int_{\RR^N}\frac{u(x)-u(y)}{|x-y|^{N+2s}}\,dy = \lim_{\varepsilon\to 0^+} (-\Delta)^s_{\varepsilon} u(x), \;\;\; x\in\RR^N,
\end{align}
provided that the limit exists. 

We notice that, if $s\in(0,1/2)$ and $u$ is a smooth function, for example bounded and Lipschitz continuous on $\RR^N$, then the integral in \eqref{fl} is in fact not really singular near $x$ (see e.g. \cite[Remark 3.1]{di2010hitchhiker}). Moreover, $\mathcal L_s^1\left(\RR^N\right)$ is the right space for which $v:= (-\Delta)^s_{\varepsilon}\, u$ exists for every $\varepsilon > 0$, $v$ being also continuous at the continuity points of $u$ (see \cite{warma2015fractional}).

We also mention that the fractional Laplace operator is a pseudo-differential operator with symbol $|\xi|^{2s}$ (see, e.g., \cite[Proposition 3.3]{di2010hitchhiker}). For more details on the fractional Laplacian we refer to \cite{di2010hitchhiker,valdinoci2009from,warma2015fractional} and the references therein.

Let us now define the function spaces in which we are going to work. It is well-known (see, e.g., \cite{di2010hitchhiker}) that the natural functional setting for problems involving the fractional Laplacian is the one of the fractional Sobolev spaces. Since these spaces are not as familiar as the classical integral order ones, for the sake of completeness, we recall here their definition and some fundamental properties. 

Given $s\in(0,1)$ and $\Omega\subset\RR^N$ a regular enough domain, the fractional Sobolev space $H^s(\Omega)$ is defined as
\begin{align*}
	H^s(\Omega):= \left\{u\in L^2(\Omega)\,:\, \frac{|u(x)-u(y)|}{|x-y|^{\frac N2+s}}\in L^2(\Omega\times\Omega) \right\}.
\end{align*}
It is classical that this is a Hilbert space, endowed with the norm (derived from the scalar product)
\begin{align*}
	\norm{u}{H^s(\Omega)} := \Big(\norm{u}{L^2(\Omega)}^2 + |u|_{H^s(\Omega)}^2\,\Big)^{\frac 12},
\end{align*}
where the term 
\begin{align*}
	|u|_{H^s(\Omega)}:= \left(\int_\Omega\int_\Omega \frac{|u(x)-u(y)|^2}{|x-y|^{N+2s}}\,dxdy\right)^{\frac 12}
\end{align*}
is the so-called Gagliardo semi-norm of $u$. Let us now denote 
\begin{align*}
	H^s_0(\Omega):= \overline{C^{\infty}_0(\Omega)}^{\,H^s(\Omega)}
\end{align*}
the closure of the continuous infinitely differentiable functions with compact support in $\Omega$ with respect to the $H^s(\Omega)$-norm. The following facts are well-known.
\begin{itemize}
	\item[$\bullet$] For $0<s\leq 1/2$, the identity $H^s_0(\Omega) = H^s(\Omega)$ holds. This is because, in this case, the $C^{\infty}_0(\Omega)$ functions are dense in $H^s(\Omega)$ (see, e.g., \cite[Theorem 11.1]{lions1968problemes}).
	
	\item[$\bullet$] For $1/2<s<1$, we have $H^s_0(\Omega)=\left\{ u\in H^s(\RR^N)\,:\,u=0\textrm{ in } \Omega^c\right\}$ (see \cite{fiscella2015density}).
\end{itemize}

Finally, in what follows, we will indicate with $H^{-s}(\Omega)=\left(H^s(\Omega)\right)^\ast$ the dual of $H^s(\Omega)$ with respect to the pivot space $L^2(\Omega)$.

We now present a couple of results which will be used in the rest of the paper. Let us consider the Dirichlet problem associated to the fractional Laplace operator
\begin{align}\label{fl dirichlet pb}
	\left\{\begin{array}{rl}
		\fl{s}{u}=g, & x\in\Omega,
		\\ u\equiv 0, & x\in\Omega^c.
	\end{array}\right.
\end{align}

By means of the Lax-Milgram Lemma (see, e.g., \cite[Proposition 2.1]{biccari2017local}), it is not difficult to see that, if $g\in H^{-s}(\Omega)$, then \eqref{fl dirichlet pb} admits a unique solution $u\in H_0^s(\Omega)$ satisfying the norm estimate
\begin{align}\label{Hs_norm}
	\norm{u}{H^s(\Omega)}\leq \C\norm{g}{H^{-s}(\Omega)},
\end{align}
where $\C$ is a positive constant independent of $g$. Moreover, in \cite[Theorem 1.3]{biccari2017local}, the following local regularity property for \eqref{fl dirichlet pb} has been obtained.
\begin{proposition}\label{local_reg_prop}
Let $g\in H^{-s}(\Omega)$ and let $u\in H^s_0(\Omega)$ be the unique weak solution to the Dirichlet problem \eqref{fl dirichlet pb}. If $g\in L^2(\Omega)$, then $u\in H^{2s}_{loc}(\Omega)$.
\end{proposition}	

Notice that, when $s\geq 1/2$, Proposition \ref{local_reg_prop} and classical embedding results for the fractional Sobolev spaces (see, e.g., \cite[Proposition 2.1]{di2010hitchhiker}) yield that the solution of \eqref{fl dirichlet pb} is in $H^1_{loc}(\Omega)$. 

In addition to that, in \cite[Proposition 1.6]{ros2014pohozaev} the following has been proved.

\begin{proposition}\label{fl pohozaev prop}
Let $\Omega$ be a bounded $C^{1,1}$ domain of $\RR^N$, $s\in (0,1)$ and $\delta(x)=\mbox{dist}(x,\partial\Omega)$, with $x\in\Omega$, be the distance of a point $x$ from $\partial\Omega$. Let $u\in H_0^s(\Omega)$ satisfy the following:
\begin{itemize}
	\item[(i)]$u\in C^s(\RR^N)$ and, for every $\beta\in [s,1+2s)$, $u$ is of class $C^{\beta}(\Omega)$ and  
	\begin{align*}
		[u]_{C^{\beta}(\{x\in\Omega\vert\delta (x)\ge\rho\})}\le \mathcal C\rho^{s-\beta}, \;\;\;\textrm{ for all } \rho\in (0,1).
	\end{align*} 
	\item[(ii)] The function $u/\delta^s\vert_{\Omega}$ can be continuously extended to $\overline{\Omega}$. Moreover, there exists $\gamma\in (0,1)$ such that $u/\delta^s\in C^{\gamma}(\overline{\Omega})$. In addition, for all $\beta\in [\gamma, s+\gamma]$  it holds the estimate 
	\begin{align*}
		[u/\delta^s]_{C^{\beta}(\{x\in\Omega\vert\delta (x)\ge\rho\})}\le \mathcal C\rho^{\gamma-\beta}, \;\;\;\textrm{ for all } \rho\in (0,1).
	\end{align*}
	\item[(iii)]$\fl{s}{u}$ is point-wise bounded in $\Omega$.
\end{itemize}
Then, the following identity holds
\begin{align}\label{fl pohozaev}
	\int_{\Omega}(x\cdot\nabla u)\fl{s}{u}\,dx=\frac{2s-N}{2}\int_{\Omega}u\fl{s}{u}\,dx-\frac{\Gamma(1+s)^2}{2}\int_{\partial\Omega}\left(\frac{u}{\delta^{s}}\right)^2(x\cdot\nu)\,d\sigma,
\end{align}
where $\nu$ is the unit outward normal to $\partial\Omega$ at $x$ and $\Gamma$ is the Gamma function.
\end{proposition}

In the Proposition \ref{fl pohozaev prop}, following the notation introduced in \cite{ros2014dirichlet,ros2014extremal,ros2014pohozaev}, $C^{\beta}(\Omega)$ with $\beta>0$ indicates the space $C^{k,\,\beta'}(\Omega)$, where $k$ is the greatest integer such that $k<\beta$ and $\beta'=\beta-k$.

Finally, in the remaining of our work we will also need the following technical results. 

\begin{proposition}
Let $u$,$v$ be two functions in $H^s_0(\Omega)$. Then, it holds the integration formula
\begin{align}\label{integration}
	\int_{\Omega}v\fl{s}{u}dx=\int_{\RR^N}\fl{\frac{s}{2}}{u}\fl{\frac{s}{2}}{v}dx=\int_{\Omega}u\fl{s}{v}dx.
\end{align}
\end{proposition}

\begin{proof}
The proof of \eqref{integration} is a simple application of Plancherel's theorem and the fact that the fractional Laplacian is the pseudo-differential operator associated to the symbol $|\xi|^{2s}$, that is 
\begin{align*}
	\fl{s}{u}=\mathscr{F}^{-1}\bigg(|\xi|^{2s}\mathscr{F}u\bigg), \;\; \textrm{ for all }\xi\in\RR^N, \; u\in H^s(\RR^N).
\end{align*}
	
To avoid any issue on the domain of $\mathscr{F}$ and $\mathscr{F}^{-1}$ one shall first take $u\in C_0^{\infty}(\Omega)$ and then use density properties. We leave the details to the reader.
\end{proof}

\begin{proposition}\label{operator T lemma}
Let $\Omega\subset\RR^N$ be a bounded regular domain. For every $g,h\in H^1_0(\Omega)$, let us define  
\begin{align*}
	T(g,h):=\int_{\Omega}\overline{g}(x\cdot\nabla h)\,dx.
\end{align*}
	
Then, for all $s\in [1/2,1)$ there exists a positive constant $\mathcal C$, depending only on $N$, $s$ and $\Omega$, such that
\begin{align}\label{operator T estimate 1}
	|T(g,h)|\leq \C\,\norm{g}{H^{1-s}_0(\Omega)}\norm{h}{H^s_0(\Omega)}
\end{align}
and
\begin{align}\label{operator T estimate 2}
	|T(g,h)|\leq \C\,\norm{g}{H^s_0(\Omega)}\norm{h}{H^s_0(\Omega)}.
\end{align}
\end{proposition}

\begin{proof}
The result follows by a simple interpolation argument. First of all, since $\Omega$ is bounded, we have 
\begin{align}\label{operator T estimate 5}
	\left|\int_{\Omega}\overline{g}(x\cdot\nabla h)\,dx\,\right| \leq \C_1(\Omega)\norm{g}{L^2(\Omega)}\norm{h}{H_0^1(\Omega)}
\end{align}
Moreover, integrating by parts and using Poincar\'e's inequality we get
\begin{align}\label{operator T estimate 6}
	\nonumber\left|\int_{\Omega}\overline{g}(x\cdot\nabla h)\,dx\,\right| &= \left|\int_{\Omega}\left(\nabla\overline{g}\cdot xh+N\overline{g}h\right)\,dx\,\right| 
	\\[5pt]
	&\leq \C(\Omega)\norm{g}{H_0^1(\Omega)}\norm{h}{L^2(\Omega)}+N\norm{g}{L^2(\Omega)}\norm{h}{L^2(\Omega)}\leq \C_2(N,\Omega)\norm{g}{H_0^1(\Omega)}\norm{h}{L^2(\Omega)}.
\end{align}
	
From \eqref{operator T estimate 5} and \eqref{operator T estimate 6} we have that $T\in\mathcal L(H_0^1(\Omega)\times L^2(\Omega))\cap\mathcal L(L^2(\Omega)\times H_0^1(\Omega))$. Therefore, applying \cite[Chapter 1, Theorems 5.1, 12.2 and 12.3]{lions1968problemes} we have $T\in\mathcal L(H^{1-s}_0(\Omega)\times H^s_0(\Omega))$ and, consequently, 
\begin{align*}
	|T(g,h)| \leq \C\,\norm{g}{H^{1-s}_0(\Omega)}\norm{h}{H^s_0(\Omega)},
\end{align*}
with $\C=\C(N,s,\Omega)$. Finally, the second inequality 
\begin{align*}
	|T(g,h)| \leq \C\,\norm{g}{H^s_0(\Omega)}\norm{h}{H^s_0(\Omega)},
\end{align*}
is trivial since, for $s\geq 1/2$, by \cite[Proposition 2.1]{di2010hitchhiker}, we have $H^s_0(\Omega)\hookrightarrow H^{1-s}_0(\Omega)$ with continuous injection. 
\end{proof}

\subsection{The one-dimensional fractional Laplacian}

In this Section, we focus on the one-dimensional case and we present the main spectral properties of the fractional Laplacian. For the rest of this sub-section, let $\Omega=(-1,1)$ and denote by $(-d_x^{\,2})^s_D$ the self-adjoint operator on $L^2(-1,1)$ associated with the closed and bilinear form
\begin{align*}
	\mathcal E(u,v)=\frac{c_{1,s}}{2}\int_{\mathbb R}\int_{\mathbb R}\frac{(u(x)-u(y))(v(x)-v(y))}{|x-y|^{1+2s}}\;dxdy,\;\;u,v\in H_0^s(-1,1).
\end{align*}
More precisely,
\begin{align*}
	D((-d_x^{\,2})^s_D)=\Big\{u\in H_0^s(-1,1):\; (-d_x^{\,2})^su\in L^2(-1,1)\Big\},\quad (-d_x^{\,2})^s_Du=(-d_x^{\,2})^su \;\mbox{ in }\;(-1,1).
\end{align*}
We refer to \cite{claus2019realization} for a rigorous and precise definition of $(-d_x^{\,2})^s_D$.

Then $(-d_x^{\,2})^s_D$ is the realization in $L^2(-1,1)$ of the fractional Laplace operator $(-d_x^{\,2})^s$ with the zero Dirichlet exterior condition $u=0$ in $(-1,1)^c$. It is well-known (see e.g. \cite{servadei2014on}) that $(-d_x^{\,2})^s_D$ has a compact resolvent and its eigenvalues form a non-decreasing sequence of real numbers
\begin{align*}
	0<\lambda_1\leq\lambda_2\leq\cdots\leq\lambda_k\leq\cdots
\end{align*}
satisfying $\lim_{k\to+\infty}\lambda_k=+\infty$. In addition to that, the eigenvalues are of finite multiplicity.

Let $(\phi_k)_{k\geq 1}$ be the orthonormal basis of eigenfunctions associated with the eigenvalues $(\lambda_k)_{k\geq 1}$, that is,
\begin{align}\label{fl_eigen}
	\begin{cases}
		\ffl{s}{\phi_k}=\lambda_k \phi_k, & x\in(-1,1),\quad k\geq 1
		\\
		\phi_k = 0 & x\in(-1,1)^c,
	\end{cases}
\end{align}
and let $s<1/2$. Then, according to \cite[Proposition 9]{servadei2013variational}, the eigenvalues $(\lambda_k)_{k\geq 1}$ can be characterized as
\begin{align}\label{eigen_char1}
	\lambda_{k+1} = \min_{\underset{\norm{\phi}{L^2(-1,1)}=1}{\phi\in\mathbb{P}_{k+1}}} \int_{\RR}\int_{\RR} \frac{|\phi(x)-\phi(y)|^2}{|x-y|^{1+2s}}\,dxdy = \min_{\underset{\norm{\phi}{L^2(-1,1)}=1}{\phi\in\mathbb{P}_{k+1}}} [\phi]_{H^s(\RR)}^2, \;\; \textrm{ for all } k\geq 1
\end{align}
or, equivalently,
\begin{align}\label{eigen_char2}
	\lambda_{k+1} = \min_{\phi\in\mathbb{P}_{k+1}\setminus\{0\}} \frac{\displaystyle\int_{\RR}\int_{\RR} \frac{|\phi(x)-\phi(y)|^2}{|x-y|^{1+2s}}\,dxdy}{\displaystyle\int_{-1}^1 |\phi|^2\,dx} = \min_{\phi\in\mathbb{P}_{k+1}\setminus\{0\}} \frac{[\phi]_{H^s(\RR)}^2}{\norm{\phi}{L^2(-1,1)}^2}, \; \textrm{ for all } k\geq 1.
\end{align}
Here $\mathbb{P}_{k+1}$ denotes the space
\begin{align*}
	\mathbb{P}_{k+1} := \Big\{\phi\in L^2(-1,1)\,:\, \phi = 0 \textrm{ a.e. in }(-1,1)^c,\;\langle \phi,\phi_j\rangle_{L^2(-1,1)}=0\textrm{ for all }j=1,\ldots,k\Big\}.
\end{align*}

Moreover, the eigenfunction $\phi_{k+1}$ is precisely the element of $\mathbb{P}_{k+1}$ attaining the minimum in \eqref{eigen_char1}, that is,
\begin{align*}
	\lambda_{k+1} = [e_{k+1}]_{H^s(\RR)}^2, \;\; \textrm{ for all } k\geq 1.
\end{align*}

To the best of our knowledge, characterizations of the spectrum of the fractional Laplacian analogous to \eqref{eigen_char1}, \eqref{eigen_char2} are not available when $s\geq 1/2$. Nevertheless, we have the following asymptotic result of the eigenvalues, whose proof is contained in \cite[Proposition 3]{kwasnicki2012eigenvalues}, that will be needed in our further analysis.

\begin{lemma}\label{lemm}
Let $1/2\leq s<1$, $\Omega=(-1,1)$ and $(\lambda_k)_{k\geq 1}$ be the eigenvalues of $(-d_x^2)^s_D$. Then the following assertions hold.
	
\begin{enumerate}
	\item[(a)] The eigenvalues $(\lambda_k)_{k\geq 1}$ are simple.
		
	\item[(b)] There is a constant $\gamma=\gamma(s)\ge \pi/2$ such that for $k$ large enough,
	\begin{align}\label{Gap}
		\lambda_{k+1}-\lambda_k\geq \gamma.
	\end{align}
\end{enumerate}
\end{lemma}

\subsection{Well-posedness}

Let us conclude this section by proving the existence and uniqueness of solutions for the fractional Schr\"odinger equation \eqref{fr_schr_control}. 

To this end, let us denote by $\mathcal A$ the realization of $\fl{s}{}$ in $L^2(\Omega)$ with zero Dirichlet exterior condition, i.e., the operator $\mathcal A:\mathcal{D}(\mathcal A)\rightarrow L^2(\Omega)$ defined as 
\begin{align}\label{operator_A}
	\mathcal{D}(\mathcal A)=\left\{u\in H_0^s(\Omega)\,\Big|\, \fl{s}{u}\in L^2(\Omega)\,\right\},\;\;\;\mathcal Au:=-\fl{s}{u}.
\end{align}

See \cite{warma2015fractional} for more details. Moreover, we mention that $\mathcal D(\mathcal A)$ has been fully characterized in \cite{grubb2015fractional}, by employing standard pseudo-differential techniques. 

It is classically known that the operator $\mathcal A$ is self-adjoint and negative. Therefore, thanks to the Stone's theorem (\cite[Chapter XI, Section 13, Theorem 1]{yosida1980functional}), $i\mathcal A$ is the generator of a one parameter $C_0$ group of unitary operators and we have the following well-posedness result (see \cite[Lemmas 4.1.1 and 4.1.5]{cazenave1998introduction}).
\begin{theorem}\label{well posedness thm}
Given $u_0\in L^2(\Omega)$ and $h\in L^2(\omega\times(0,T))$, the system \eqref{fr_schr_control} admits a unique solution
\begin{align*}
	u\in C\big([0,T];L^2(\Omega)\big).
\end{align*}
Moreover, if $u_0\in\mathcal{D}(\mathcal A)$ then
\begin{align*}
	u\in C\big([0,T];\mathcal{D}(\mathcal A)\big)\cap C^1\big([0,T];L^2(\Omega)\big).
\end{align*}
\end{theorem}

\section{The one dimensional controllability problem}\label{1d_sec}
This section is devoted to the analysis of the controllability of \eqref{fr_schr_control} in the one-dimensional case. In particular, we are going to prove here that, when $N=1$, the fractional Schr\"odinger equation \eqref{fr_schr_control} is null controllable by means of an interior control $h$ supported in an open subset $\omega\subset(-1,1)$ if and only if $s\geq 1/2$. Thus, the main result of this section will be the following Theorem.
\begin{theorem}\label{schr sharp} 
Let $\omega\subset (-1,1)$ and $T>0$. Let us consider the following control problem for the one-dimensional fractional Schr\"odinger equation on the interval $(-1,1)$
\begin{align}\label{schr 1d}
	\left\{\begin{array}{ll}
		iu_t+\ffl{s}{u}=g\chi_\omega, & (x,t)\in (-1,1)\times(0,T),
		\\ 
		u\equiv 0, & (x,t)\in (-1,1)^c\times(0,T),
		\\ 
		u(x,0)=u_0(x), & x\in (-1,1).
	\end{array}\right.
\end{align}
\begin{enumerate}
	\item[(i)] If $s\in(1/2,1)$, for any $T>0$ and any given $u_0\in L^2(-1,1)$, there exists $g\in L^2(\omega\times(0,T))$ such that the solution $u$ of \eqref{schr 1d} satisfies $u(x,T)=0$.
	\item[(ii)] If $s=1/2$, for any $T>4$ and any given $u_0\in L^2(-1,1)$, there exists $g\in L^2(\omega\times(0,T))$ such that the solution $u$ of \eqref{schr 1d} satisfies $u(x,T)=0$.
\end{enumerate}
\end{theorem}

\begin{proof}[Proof of Theorem \ref{schr sharp}]
We recall that, by means of HUM, the controllability of \eqref{schr 1d} is equivalent to the observability of the adjoint equation
\begin{align}\label{schr 1d adj}
	\left\{\begin{array}{ll}
		iv_t+\ffl{s}{v}=0, & (x,t)\in (-1,1)\times(0,T),
		\\ 
		v\equiv 0, & (x,t)\in (-1,1)^c\times(0,T),
		\\ 
		v(x,0)=v_0(x), & x\in (-1,1).
	\end{array}\right.
\end{align}
In other words, it will be enough to show that, for all solution $v$ to \eqref{schr 1d adj}, it holds the inequality
\begin{align}\label{obs 1d}
	\norm{v_0}{ L^2(-1,1)}^2\leq \C\int_0^T\norm{v(t)}{ L^2(\omega)}\,dt.
\end{align}
	
To this end, let us write the solution of \eqref{schr 1d adj} in terms of the spectrum of the Dirichlet fractional Laplacian on $(-1,1)$, that is, 
\begin{align}\label{vk_eigen}
	v(x,t)=\sum_{k\geq 1} v_ke^{i\lambda_k t}\phi_k(x),
\end{align}
with 
\begin{align*}
	v_k=\int_{-1}^1v_0(x)\phi_k(x)\,dx.
\end{align*}
	
From \eqref{vk_eigen}, by using the orthonormality of the eigenfunctions in $ L^2(-1,1)$, it is immediate to see that the observability inequality \eqref{obs 1d} may be rewritten as
\begin{align}\label{obs 1d ingham_prel}
	\sum_{k\geq 1} |v_k|^2\leq \C\int_0^T\int_\omega \left|\sum_{k\geq 1}v_ke^{i\lambda_k t}\phi_k(x)\,\right|^2\,dxdt.
\end{align}
	
In order to obtain \eqref{obs 1d ingham_prel} we will employ classical Ingham's techniques (see, e.g., \cite{ingham1936some}, \cite[Section 4]{micu2004introduction} or \cite[Chapter 8, Theorem 8.1.1]{tucsnak2009observation}). In particular, in \cite[Theorem 4.3]{micu2004introduction} it has been proven that, if there is a positive gap between the eigenvalues, namely
\begin{align}\label{gap}
	\liminf_{k\to +\infty}(\lambda_{k+1}-\lambda_k)=\gamma_{\infty}>0,
\end{align}
then, for any finite real time $T>2\pi/\gamma_\infty$ and any finite sequence $\{a_k\}_{k\geq 1}$, it holds the inequality
\begin{align}\label{obs 1d ingham}
	\sum_{k\geq 1} |a_k|^2\leq \C(T)\int_0^T \left|\sum_{k\geq 1}a_ke^{i\lambda_k t}\,\right|^2\,dt,
\end{align}
where the constant $\C(T)$ depends only on the time horizon $T$.
	
Moreover, employing the asymptotic results on the spectrum of the fractional Laplacian contained in the papers \cite{kulczycki2010spectral,kwasnicki2012eigenvalues}, we can easily show that, for any sub-interval $\omega\subset(-1,1)$, there exists a positive constant $\beta>0$ such that it holds the lower estimate
\begin{align}\label{eigen_l2_est}
	\norm{\phi_k}{L^2(\omega)}\geq\beta>0.
\end{align} 
	
We mention that \eqref{eigen_l2_est} can be also obtained as a direct consequence of \cite[Lemma 3.2]{biccari2019controllability}, where the same inequality in the $L^1$ setting has been proved. 
	
Let us now fix $x\in(-1,1)$. Taking $a_k = v_k\phi_k(x)$ in \eqref{obs 1d ingham} we obtain that there is a constant $\C(T)$, independent of $x$, such that it holds the inequality
\begin{align}\label{obs 1d ingham_x}
	\sum_{k\geq 1} |v_k\phi_k(x)|^2\leq \C(T)\int_0^T \left|\sum_{k\geq 1}v_ke^{i\lambda_k t}\phi_k(x)\,\right|^2\,dt.
\end{align}
Then, integrating \eqref{obs 1d ingham_x} over $\omega$ and using the estimate \eqref{eigen_l2_est}, we get that
\begin{align*}
	\beta\sum_{k\geq 1} |v_k|^2 \leq \int_\omega \sum_{k\geq 1} |v_k\phi_k(x)|^2\,dx\leq \C(T)\int_\omega\int_0^T \left|\sum_{k\geq 1}v_ke^{i\lambda_k t}\phi_k(x)\,\right|^2\,dtdx,
\end{align*}
which is clearly equivalent to \eqref{obs 1d ingham_prel}.
	
Hence, in order to conclude our proof, it only remains to check whether \eqref{obs 1d ingham} holds in our case. In other words, we shall see if the gap condition \eqref{gap} is satisfied. 
	
To this end, we recall that from \cite[Theorem 1]{kwasnicki2012eigenvalues} (see also Lemma \ref{lemm}) we know that the eigenvalues of the Dirichlet fractional Laplacian on $(-1,1)$ are given by 
\begin{align*}
	\lambda_k=\left(\frac{k\pi}{2}-\frac{(2-2s)\pi}{8}\right)^{2s}+O\left(\frac{1}{k}\right), \;\;\; \textrm{ as } k\to +\infty,
\end{align*}
while \cite[Propositon 3]{kwasnicki2012eigenvalues} ensures that $\lambda_k$ are simple if $s\geq 1/2$ (see Figures \ref{eigen_plot} and \ref{gap_fig}). In particular, we can readily check that
\begin{align*}
	\gamma_{\infty}=\liminf_{k\to +\infty}(\lambda_{k+1}-\lambda_k)=\begin{cases}
		0, & s<1/2
		\\
		\pi/2, &  s=1/2
		\\
		+\infty, & s>1/2.
	\end{cases}
\end{align*}
	
\begin{figure}[h]
	\centering
	\begin{minipage}{0.45\textwidth}
		\centering
		\includegraphics[scale=0.45]{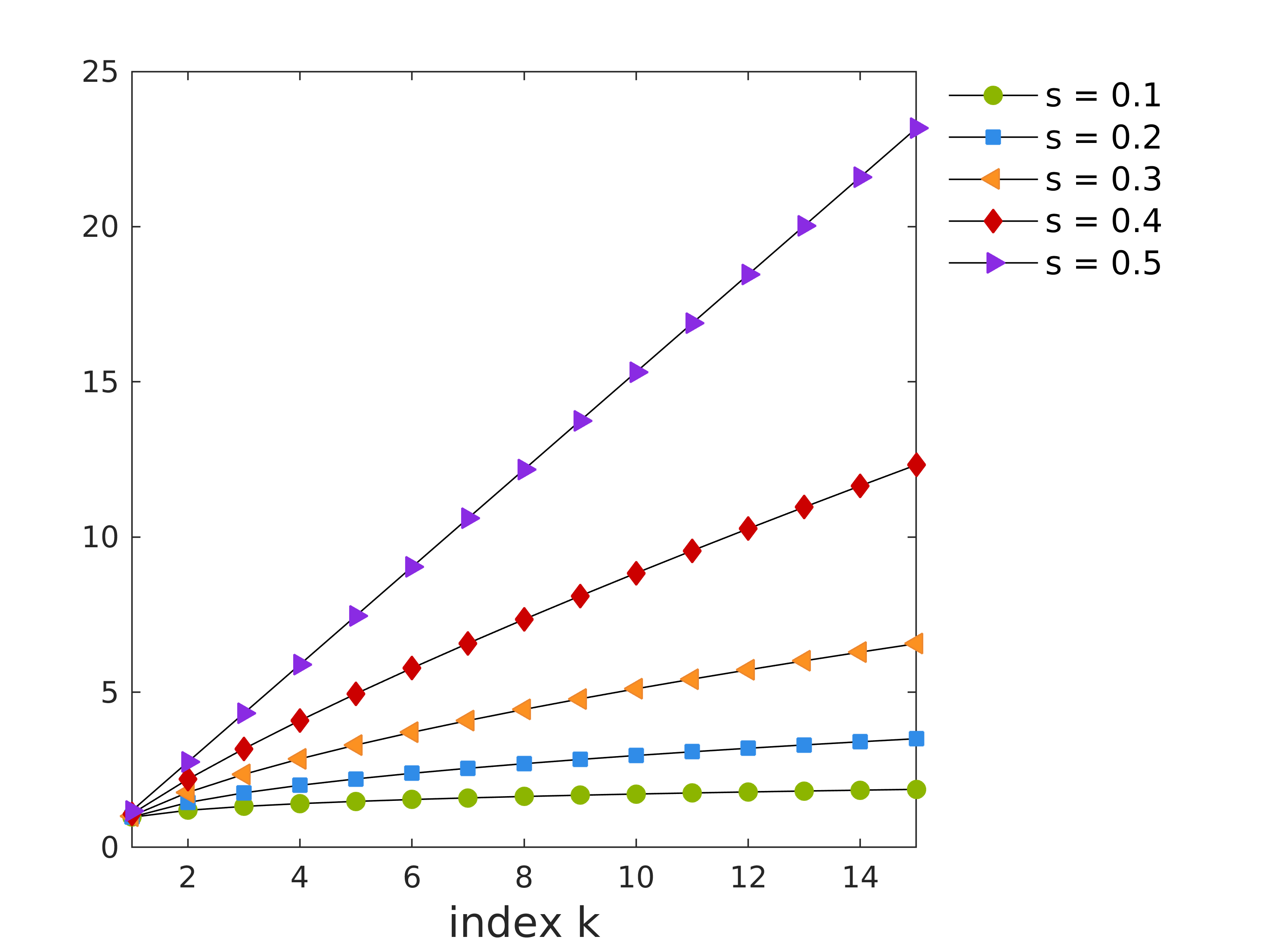}
	\end{minipage}
	\begin{minipage}{0.45\textwidth}
		\centering
		\includegraphics[scale=0.45]{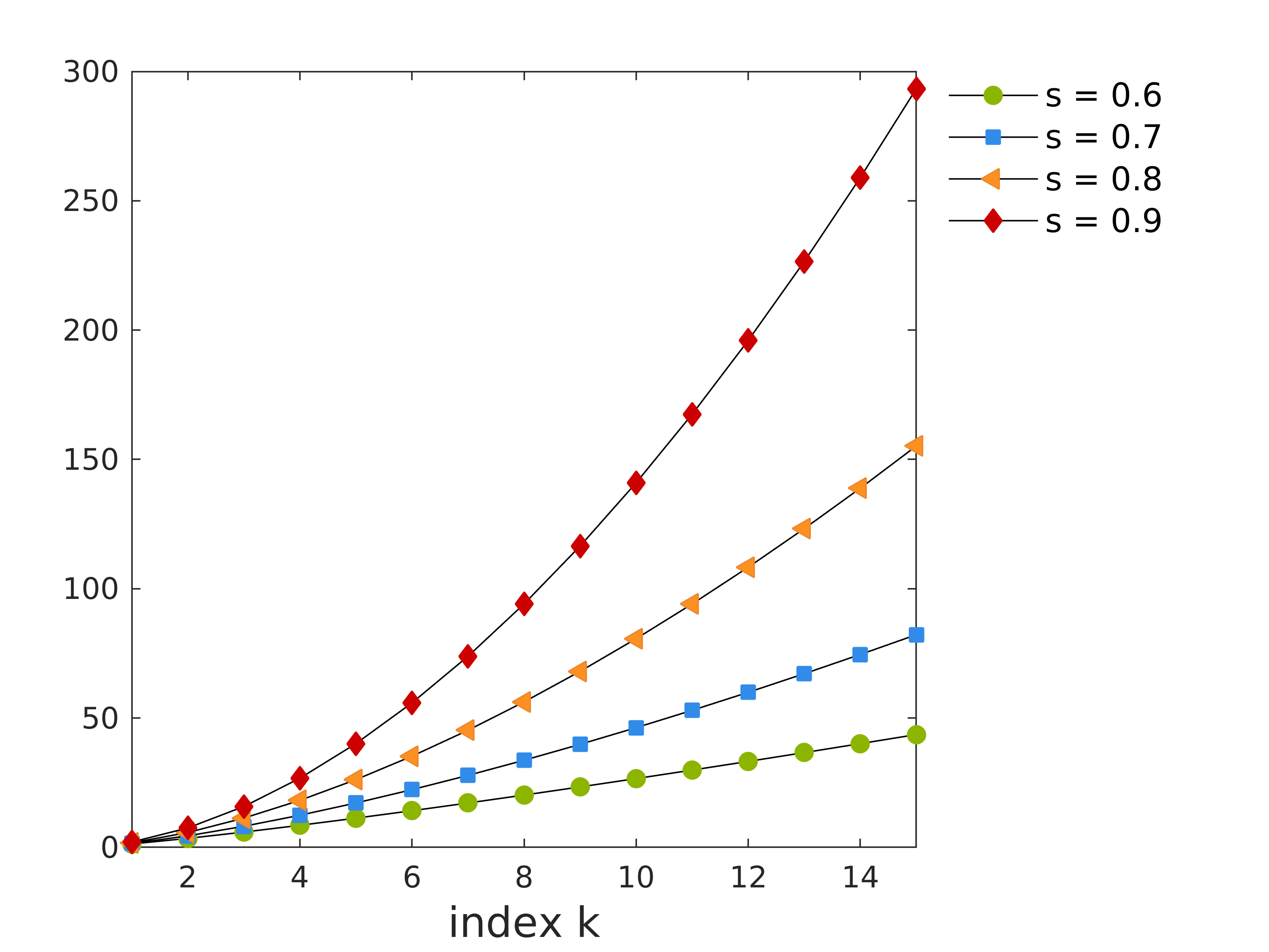}
	\end{minipage}
	\caption{First 15 eigenvalues of the Dirichlet fractional Laplacian $\ffl{s}{}$ on $(-1,1)$ for $s\in (0,1/2]$ (left) and $s\in(1/2,1)$ (right).}
	\label{eigen_plot}
\end{figure}
	
\begin{figure}[h]
	\centering
	\begin{minipage}{0.45\textwidth}
		\centering
		\includegraphics[scale=0.45]{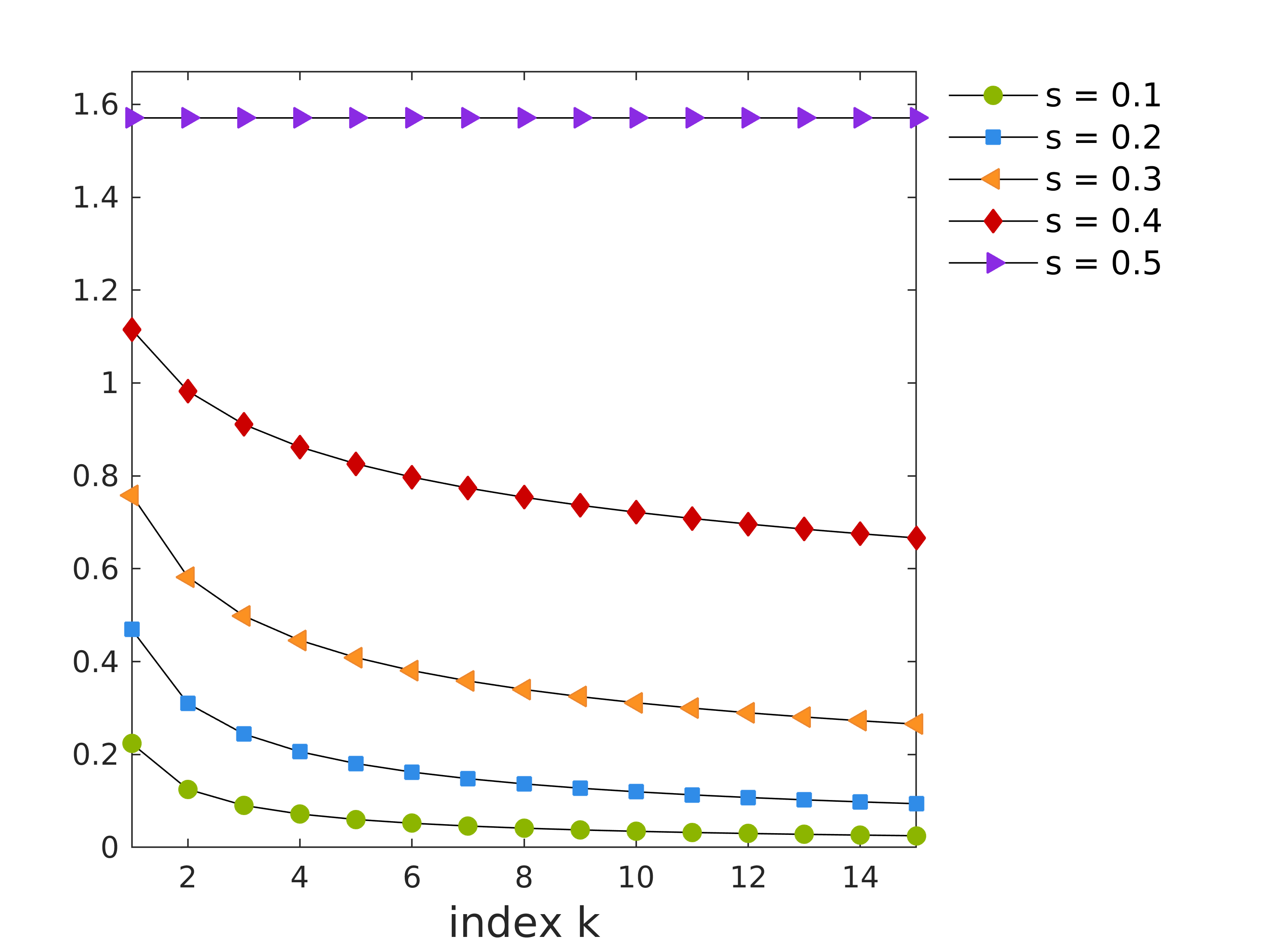}
	\end{minipage}
	\begin{minipage}{0.45\textwidth}
		\centering
		\includegraphics[scale=0.45]{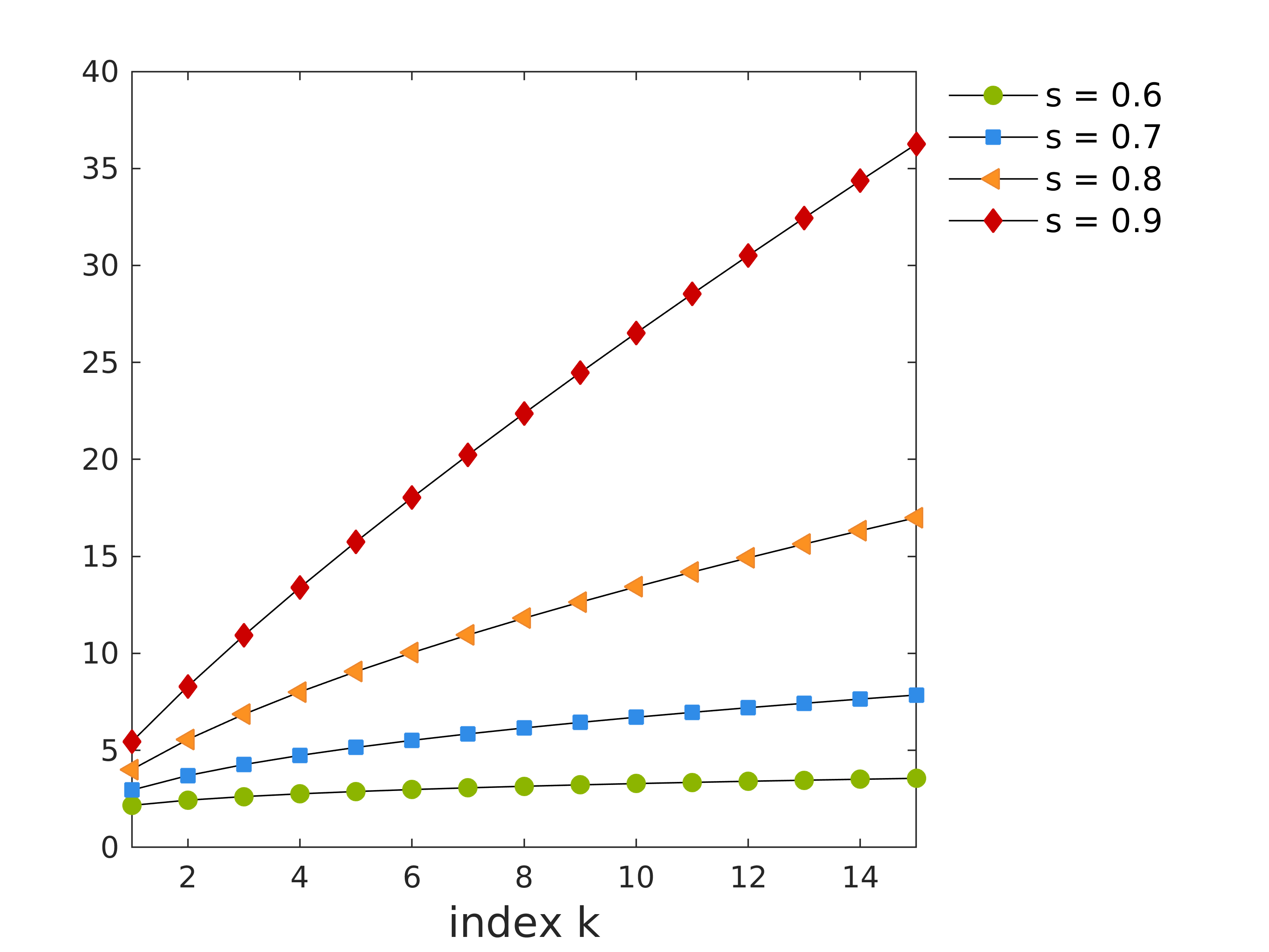}
	\end{minipage}
	\caption{Gap between the first 15 eigenvalues of the Dirichlet fractional Laplacian $\ffl{s}{}$ on $(-1,1)$ for $s\in(0,1/2]$ (left) and $s\in(1/2,1)$ (right).}\label{gap_fig}
\end{figure}
	
This, in turn, implies that \eqref{obs 1d ingham} (and, consequently, \eqref{obs 1d}) holds for any $T>4/\pi$, if $s=1/2$, and for all $T>0$, if $s>1/2$. Our proof is then concluded.
\end{proof}

\section{Fractional Schr\"odinger equation}\label{fl_schr_sec}

We address now the multi-dimensional case, in which the observability inequality \eqref{schr obs} will be obtained by employing the nowadays classical multiplier method. To this end, a series of preliminary technical results will be needed. 

\subsection{Pohozaev-type identity}

In this Section, we introduce one of the main tools that are needed in order to prove the controllability Theorem \ref{schr_control_thm}: a Pohozaev-type identity for the solution of our fractional Schr\"odinger equation. This identity will be obtained employing the classical multiplier method (see, e.g., \cite{komornik1994exact,lions1988controlabilite,lions1988exact}), joint with the Pohozaev identity for the fractional Laplacian that we introduced in Proposition \ref{fl pohozaev prop} (see also \cite[Proposition 1.6]{ros2014pohozaev}). 

\begin{proposition}\label{schr_pohozaev_prop}
Let $\Omega$ be a bounded $C^{1,1}$ domain of $\RR^N$, $s\in [1/2,1)$ and $\delta(x)$ be the distance of a point $x$ from $\partial\Omega$. Moreover, let $\mathcal D(\mathcal A)$ be defined as in \eqref{operator_A}. For any $h\in L^2(\Omega\times(0,T))$ and any initial datum $u_0\in\mathcal D(\mathcal A)$, let $u$ be the corresponding solution of \eqref{fr_schr_control}. Then, the following identity holds
\begin{align}\label{schr_pohozaev_nh}
	\nonumber\Gamma(1+s)^2\int_{\Sigma}\left(\frac{|u|}{\delta^s}\right)^2(x\cdot\nu)\,d\sigma dt =&\, 2s\int_0^T\left \|\fl{\frac{s}{2}}{u(t)}\right \|^2_{ L^2(\RR^N)}\,dt+ \Im \int_{\Omega}\overline{u}(x\cdot\nabla u)\,\big |_0^T\,dx 
	\\ 
	&- \Re\int_0^T\int_{\omega} h\,\Big(N\overline{u}+2x\cdot\nabla\overline{u}\,\Big)\,dxdt,
\end{align}
where $\nu$ is the unit outward normal to $\partial\Omega$ at $x$, $\Gamma$ is the Gamma function and $\Sigma:=\partial\Omega\times [0,T]$.
\end{proposition}

\begin{proof}
	
For proving Proposition \ref{schr_pohozaev_prop}, we are going to apply the classical method of multipliers (see \cite{komornik1994exact,lions1988controlabilite,lions1988exact}), joint with the Pohozaev identity for the fractional Laplacian that we introduced in Proposition \ref{fl pohozaev prop} (see also \cite[Proposition 1.6]{ros2014pohozaev}). 
	
However, this mentioned identity holds under very strict regularity assumptions, which are not necessarily satisfied by the solution $u$ of \eqref{fr_schr_control}. 
	
In order to bypass this regularity issue, we can first obtain \eqref{schr_pohozaev_nh} for solutions of \eqref{fr_schr_control} corresponding to an initial datum $u_0$ given as a linear combination of a finite number of eigenfunctions of the Dirichlet fractional Laplacian on $\Omega$. Indeed, it follows from \cite[Proposition 4]{servadei2013brezis} and \cite[Corollary 1.6]{ros2014dirichlet} that the Pohozaev identity holds in this case. 
	
In a second moment, we can recover the result for any finite energy solution $u$ by using a compactness argument. Since this compactness argument is a standard one, in order to abridge our presentation, in what follows we will skip it and we will always assume that we have all the required regularity to perform our computations.
	
Let us now prove the identity \eqref{schr_pohozaev_nh}. To this end, we multiply our equation \eqref{fr_schr_control} by $x\cdot\nabla\overline{u}+(N/2)\overline{u},$ we take the real part, and we integrate over $Q$, obtaining
\begin{align*}
	\Re\int_0^T\int_\omega h\left(x\cdot\nabla\overline{u}+\frac N2\overline{u}\right)\,dxdt =& \underbrace{\Re\int_Q \fl{s}{u}(x\cdot\nabla\overline{u})\,dxdt}_{A_1}\;+\underbrace{\Re\int_Q\frac N2\overline{u}\fl{s}{u}\,dxdt}_{A_2}
	\\
	&+\underbrace{\Re\int_Q iu_t\left(\frac N2\overline{u}+x\cdot\nabla\overline{u}\right)\,dxdt.}_{A_3}
\end{align*}
	
We now compute the three contributions on the right-hand side separately. For the first integral, employing \eqref{fl pohozaev}, we have
\begin{align*}
	A_1 =&\int_Q\bigg\{\Big[\fl{s}{\Re(u)}\Big]\Big(x\cdot\nabla\Re(u)\Big)+\Big[\fl{s}{\Im(u)}\Big]\Big(x\cdot\nabla\Im(u)\Big)\bigg\}\,dxdt
	\\ 
	=&\,\frac{2s-N}{2}\int_Q\bigg\{\Re(u)\Re\big[\fl{s}{u}\big]+\Im(u)\Im\big[\fl{s}{u}\big]\bigg\}\,dxdt
	\\
	&-\frac{\Gamma(1+s)^2}{2}\int_{\Sigma}\left[\left(\frac{\Re(u)}{\delta^s}\right)^2+\left(\frac{\Im(u)}{\delta^s}\right)^2\right](x\cdot\nu)\,d\sigma dt
	\\
	=&\,\frac{2s-N}{2}\,\Re\int_Q \overline{u}\fl{s}{u}\,dxdt-\frac{\Gamma(1+s)^2}{2}\int_{\Sigma}\left(\frac{|u|}{\delta^s}\right)^2(x\cdot\nu)\,d\sigma dt.
\end{align*}
Thus, using also \eqref{integration}, we obtain
\begin{align*}
	A_1+A_2 &= s\,\Re\int_Q \overline{u}\fl{s}{u}\,dxdt-\frac{\Gamma(1+s)^2}{2}\int_{\Sigma}\left(\frac{|u|}{\delta^s}\right)^2(x\cdot\nu)\,d\sigma dt
	\\
	&= s\,\int_0^T \norm{\fl{\frac s2}{u(t)}}{L^2(\RR^N)}^2\,dt-\frac{\Gamma(1+s)^2}{2}\int_{\Sigma}\left(\frac{|u|}{\delta^s}\right)^2(x\cdot\nu)\,d\sigma dt.
\end{align*}
	
In order to compute the integral $A_3$, we observe that, by considering the function $\psi(x):=|x|^2/4$ we have
\begin{align*}
	\nabla\psi=\frac{x}{2}, \quad \Delta\psi=\frac{N}{2}.
\end{align*}
Hence
\begin{align*}
	A_3 &= \displaystyle\Re\int_Q i(u)_t\,\Big(\,\overline{u}\Delta\psi+2\nabla\psi\cdot\nabla\overline{u}\Big)\,dxdt =-\Im\int_Q u_t\,\Big(\,\overline{u}\Delta\psi+2\nabla\psi\cdot\nabla\overline{u}\Big)\,dxdt 
	\\
	&= -\Im\int_Q\bigg\{-\nabla\Big[u_t\overline{u}\Big]\cdot\nabla\psi +2u\overline{u}\cdot\nabla\psi\bigg\}\,dxdt 
	\\
	&= -\Im\int_Q\bigg[-\overline{u}\nabla u_t\cdot\nabla\psi-u_t\nabla\overline{u} 		\cdot\nabla\psi+2u_t\nabla\overline{u}\cdot\nabla\psi\bigg]\,dxdt
	\\
	&=  \Im\int_Q\bigg[\overline{u}\nabla u_t\cdot\nabla\psi- u_t\nabla\overline{u}\cdot\nabla\psi\bigg]\,dxdt = \Im\int_Q\partial_t\bigg[\overline{u}\nabla u\cdot\nabla\psi\bigg]\,dxdt = \Im\int_Q\partial_t\left[\frac{\overline{u}}{2}(x\cdot\nabla u)\right]\,dxdt 
	\\
	&= \Im\int_{\Omega}\frac{\overline{u}}{2}(x\cdot\nabla u)dx\,\big |_0^T.
\end{align*}
Adding the components previously obtained, we finally get \eqref{schr_pohozaev_nh}. Our proof is then concluded.
\end{proof}

\subsection{Boundary observability}

We now use \eqref{schr_pohozaev_nh}, applied to the solution of \eqref{fr_schr_adj}, to obtain upper and lower estimates for the $ H^{s}_0(\Omega)$ norm of the initial datum $v_0$ on $\Omega$ with respect to the boundary term appearing in the identity \eqref{fl pohozaev}. To this end, we shall first prove the following technical result.
\begin{lemma}\label{cu_lemma}
Let $v\in L^\infty(0,T; H^{s}_0(\Omega))\cap W^{1,\infty}(0,T, H^{-s}(\Omega))$ be a solution of the adjoint equation \eqref{fr_schr_adj} such that 
\begin{align*}
	\frac{|v|}{\delta^s}=0\,\,\textrm{ on }\Sigma.
\end{align*}
Then, $v\equiv 0$.
\end{lemma}

\begin{proof}
For simplicity of notation, let us define 
\begin{align*}
	X:=L^\infty(0,T; H^{s}_0(\Omega))\cap W^{1,\infty}(0,T, H^{-s}(\Omega))
\end{align*} 
and, for every $v\in X$, let us consider the space 
\begin{align*}
	\mathcal{V}:=\bigg\{v\in X\;\big|\,v \textrm{ solves }(\ref{fr_schr_adj}) \textrm{ and }\frac{|u|}{\delta^s}=0 \textrm{ on } \Sigma\,\bigg\}\subset X,
\end{align*}
equipped with the norm endowed by $X$. Clearly it will be enough to prove that $\mathcal{V}=\{0\}$. For doing that, we are going to proceed in two steps.
	
\paragraph{\textit{Step 1.}}
We firstly show that $\textrm{dim}(\mathcal{V})<\infty$. With this purpose, let us define 
\begin{align*}
	z:=iv_t.
\end{align*}
	
With the same arguments employed in the proof of \cite[Appendix I, Lemma 2.1]{lions1988controlabilite}, we can immediately show that $z\in X$. Moreover, it is easy to check that $z$ is also a solution of \eqref{fr_schr_adj} and that the condition $|z|/\delta^s=0$ on $\Sigma$ is satisfied. Therefore, $z\in\mathcal{V}$ and, using the results of \cite{simon1986compact}, we have that the injection
\begin{align*}
	\big\{v\in\mathcal{V}\,;\,iv_t\in\mathcal{V}\big\}\hookrightarrow\mathcal{V}
\end{align*}
is continuous and compact. This, in particular, implies that the dimension of $\mathcal{V}$ is finite.
	
\paragraph{\textit{Step 2.}} We argue by contradiction, assuming that $\mathcal{V}\neq\{0\}$. Since $\mathcal{V}$ has finite dimension, given the linear map $\Phi:\mathcal{V}\to\mathcal{V}$ defined as 
\begin{align*}
	\Phi(\psi)=i\psi_t
\end{align*} 
there exists $\lambda\in\CC$ and $\psi\in\mathcal{V}\setminus\{0\}$ such that 
\begin{align}\label{v eigen}
	i\psi_t=\lambda\psi.
\end{align}
	
Moreover, we have $\lambda\neq 0$. Indeed, if $\lambda=0$, then also $\psi_t=0$ and, since by definition $\psi$ is a solution of (\ref{fr_schr_adj}), this implies that it solves 
\begin{align*}
	\left\{\begin{array}{ll}
		\fl{s}{\psi}=0, & x\in\Omega
		\\
		\psi\equiv 0, & x\in\Omega^c.
	\end{array}\right.
\end{align*}
Hence, thanks to \eqref{fl pohozaev} and \eqref{integration}, and since $|\psi|/\delta^s=0$ on $\partial\Omega$, we have
\begin{align*}
	0 = \int_\Omega \fl{s}{\psi}(x\cdot\nabla\psi)\,dx = \frac{2s-N}{2}\int_\Omega \psi\fl{s}{\psi}\,dx -\frac{\Gamma(1+s)^2}{2}\int_{\partial\Omega} \left(\frac{\psi}{\delta^s}\right)^2\,d\sigma = \frac{2s-N}{2}\norm{\psi}{H^s(\RR^N)}^2.
\end{align*}
	
This implies $\psi\equiv 0$, which is contradictory. Now, for $\lambda\neq 0$, using again the Pohozaev identity (\ref{fl pohozaev}) and (\ref{v eigen}), we have that
\begin{align*}
	0&=\frac{\Gamma(1+s)^2}{2}\int_{\Sigma}\left(\frac{|\psi|}{\delta^s}\right)^2(x\cdot\nu)\,d\sigma dt 
	\\
	&= \frac{2s-N}{2}\Re\int_Q\overline{\psi}\fl{s}{\psi}\,dxdt-\int_Q (x\cdot\nabla\overline{\psi})\fl{s}{\psi}\,dxdt
	\\
	&= -\frac{2s-N}{2}\Re\int_Q\overline{\psi}(i\psi_t)\,dxdt+\Re\int_Q (x\cdot\nabla\overline{\psi})(i\psi_t)\,dxdt 
	\\
	&= -\frac{\lambda(2s-N)}{2}\Re\int_Q\psi\overline{\psi}\,dxdt+\lambda\Re\int_Q (x\cdot\nabla\overline{\psi})\psi\,dxdt
	\\
	&= -\frac{\lambda(2s-N)}{2}\Re\int_Q\psi\overline{\psi}\,dxdt-\frac{\lambda N}{2}\Re\int_Q \psi\overline{\psi}\,dxdt
	\\
	&= -s\lambda\norm{\psi}{ L^2(Q)}^2.
\end{align*}
	
Therefore, we have that also in this case $\psi\equiv0$, and this is again contradictory. Our proof is then concluded.
\end{proof}

With the help of Lemma \ref{cu_lemma}, we can now prove the main result of this sub-section, i.e. the following theorem.

\begin{theorem}\label{schr est thm}
Let $v_0\in\mathcal D(\mathcal A)$, where $\mathcal D(\mathcal A)$ has been defined in \eqref{operator_A}. Then, there exist two positive constants $\C_1$ and $\C_2$, depending only on $s$, $T$, $N$ and $\Omega$, such that 
\begin{enumerate}
	\item[(i)] if $s\in(1/2,1)$, then for any $T>0$ and for all $v$ solution of \eqref{fr_schr_adj} it holds
	\begin{align}\label{schr est}
		\C_1\norm{v_0}{ H^{s}_0(\Omega)}^2\leq \int_{\Sigma}\left(\frac{|v|}{\delta^s}\right)^2(x\cdot\nu)\,d\sigma dt\leq \C_2\norm{v_0}{ H^{s}_0(\Omega)}^2;
	\end{align}
	\item[(ii)] if $s=1/2$, there exists a time $T_0>0$  such that \eqref{schr est} holds for any $T>T_0$.
	\end{enumerate}
\end{theorem}

\begin{proof}
First of all, without loss of generality, we will assume that the function $v$ is smooth enough for our computations. As we did before, this fact can be justified passing through the decomposition of $v$ in the basis of the eigenfunctions $\phi_k$ and then arguing by compactness. 
	
Secondly, since $i\fl{s}{}$ is a skew-adjoint operator, for all $t\in(0,T)$ it holds 
\begin{align}\label{schr norm identities}
	&\norm{v(x,t)}{ L^2(\Omega)}=\norm{v_0}{ L^2(\Omega)}\nonumber 
	\\[7pt]
	&\norm{v(x,t)}{ H^{s}_0(\Omega)}=\norm{v_0}{ H^{s}_0(\Omega)}
	\\[7pt]
	&\PP_1\norm{v_0}{ H^{1-s}_0(\Omega)}\leq\norm{v(x,t)}{ H^{1-s}_0(\Omega)}\leq \PP_2\norm{v_0}{ H^{1-s}_0(\Omega)}.\nonumber 
\end{align}
	
Furthermore, we recall that, by the regularity obtained in the well-posedness Theorem \ref{well posedness thm}, we have that $\fl{s}{v}=-iv_t\in L^2(\Omega)$ and this fact immediately implies $v\in  H^{2s}_{loc}(\Omega)$, due to Proposition \ref{local_reg_prop} (see also \cite{biccari2017addendum,biccari2017local}). In particular, since $s\geq 1/2$ we also have $v\in  H^1_{loc}(\Omega)$. Now, considering \eqref{schr_pohozaev_nh} with $h=0$ we obtain
\begin{align}\label{schr_pohozaev_h}
	\Gamma(1+s)^2\int_{\Sigma}\left(\frac{|v|}{\delta^s}\right)^2(x\cdot\nu)\,d\sigma dt = 2s\int_0^T\norm{\fl{\frac{s}{2}}{v(t)}}{L^2(\RR^N)}^2\,dt + \Im\int_{\Omega}\overline{v}(x\cdot\nabla v)\,\big|_0^T\,dx.
\end{align}
	
For proving our result, we will apply Proposition \ref{operator T lemma} to the last term of \eqref{schr_pohozaev_h}, thus obtaining the following estimate 
\begin{align}\label{est_mintime}
	\left |\,\int_{\Omega}\overline{v}(x\cdot\nabla v)\,dx\,\right | \leq \C\,\norm{v(t)}{ H^{s}_0(\Omega)}\norm{v(t)}{ H^{1-s}_0(\Omega)}.
\end{align}
	
Therefore, it will be necessary to distinguish the two cases $s>1/2$ and $s=1/2$. Indeed, for $s>1/2$, since the $ H^{1-s}_0$ terms are lower order with respect to the $H^{s}_0$ ones, we can deal with them by applying a compactness-uniqueness argument. However for $s=1/2$, since of course the spaces $H^{1-s}_0$ and $H^{s}_0$ coincide, we have to proceed in a different way.
	
\paragraph{\textit{Case $s=1/2$.}} Employing \ref{operator T estimate 2} and \eqref{schr norm identities}, we obtain 
\begin{align*}
	\left|\,\int_{\Omega}\overline{v}(x\cdot\nabla v)\,dx\,\right |\leq \C\,\norm{v(t)}{ H^{1/2}(\Omega)}^2,
\end{align*}
Hence, from \eqref{schr_pohozaev_h} we get
\begin{align}\label{est_mintime_constant}
	(T-2\C)\norm{v_0}{ H^{1/2}(\Omega)}^2\leq\int_{\Sigma}\left(\frac{|v|}{\delta^{1/2}}\right)^2(x\cdot\nu)\,d\sigma dt\leq (T+2\C)\norm{v_0}{ H^{1/2}(\Omega)}^2.
\end{align}
Thus, finally, if $T>2\C:=T_0$, the inequalities
\begin{align*}
	\C_1\norm{v_0}{ H^{1/2}(\Omega)}^2\leq\int_{\Sigma}\left(\frac{|v|}{\delta^s}\right)^2(x\cdot\nu)\,d\sigma dt\leq \C_2\norm{v_0}{ H^{1/2}(\Omega)}^2
\end{align*}
hold with $\C_1,\,\C_2>0$. 
	
\paragraph{\textit{Case $s>1/2$.}} First of all, we have
\begin{align*}
	\Gamma(1+s)^2\int_{\Sigma}\left(\frac{|v|}{\delta^s}\right)^2(x\cdot\nu)\,d\sigma dt \leq 2sT\norm{v_0}{ H^{s}_0(\Omega)}^2+2\left |\,\int_{\Omega}\overline{v}(x\cdot\nabla v)\,dx\,\right |\leq \C_2\norm{v_0}{ H^{s}_0(\Omega)}^2,
\end{align*}
where we used again \eqref{operator T estimate 2} and \eqref{schr norm identities}.
	
Let us now prove the other estimate. By using \eqref{operator T estimate 1} and  \eqref{schr norm identities}, and applying Young's inequality, for all $\varepsilon>0$ we have 
\begin{align*}
	\left |\,\int_{\Omega}\overline{v}(x\cdot\nabla v)\,dx\,\right | \leq \C\varepsilon\,\norm{v_0}{ H^{s}_0(\Omega)}^2 + \frac{\C}{4\varepsilon}\,\norm{v_0}{H^{1-s}_0(\Omega)}^2.
\end{align*}
Thus, choosing $\varepsilon<2sT/\C$, we get that 
\begin{align*}
	(2sT-\C\varepsilon)\norm{v_0}{H^{s}_0(\Omega)}^2 \leq \Gamma(1+s)^2\int_{\Sigma}\left(\frac{|v|}{\delta^s}\right)^2(x\cdot\nu)\,d\sigma dt + \frac{\C}{4\varepsilon}\,\norm{v_0}{H^{1-s}_0(\Omega)}^2.
\end{align*}
	
We conclude by observing that, thanks to a compactness-uniqueness argument, we can prove that there exists a positive constant $\M$, not depending on $v$, such that
\begin{align}\label{cu_est}
	\norm{v_0}{H^{1-s}_0(\Omega)}^2\leq \M\int_{\Sigma}\left(\frac{|v|}{\delta^s}\right)^2(x\cdot\nu)\,d\sigma dt.
\end{align}
	
Indeed, let us assume that the previous inequality does not hold. This implies that there exists a sequence $\{v^j\}_{j\in\NN}\subset  H^{1-s}_0(\Omega)$ of solutions of \eqref{fr_schr_adj} such that 
\begin{align}\label{cu_norm}
	\norm{v^j(0)}{ H^{1-s}_0(\Omega)}=1, \;\;\;\textrm{ for all } j\in\NN
\end{align}
and 
\begin{align}\label{cu_int}
	\lim_{j\to+\infty}\int_{\Sigma}\left(\frac{|v^j|}{\delta^s}\right)^2(x\cdot\nu)\,d\sigma dt=0.
\end{align}
	
From \eqref{cu_norm} we deduce that $\{v^j(0)\}_{j\in\NN}$ is bounded in $ H^{s}_0(\Omega)$ and then, from \eqref{fr_schr_adj} and \eqref{schr norm identities}, $\{v^j\}_{j\in\NN}$ is bounded in $L^\infty(0,T; H^{s}_0(\Omega))\cap W^{1,\infty}(0,T, H^{-s}(\Omega))$. Therefore, by extracting a sub-sequence, that we will still note by $\{v^j\}_{j\in\NN}$, we have
\begin{align*}
	\left\{\begin{array}{lll}
		v^j\rightharpoonup v & \textrm{ in }L^\infty(0,T; H^{s}_0(\Omega)) &\textrm{ weakly *},
		\\
		\partial_t v^j\rightharpoonup \partial_tv & \textrm{ in }L^\infty(0,T; H^{-s}(\Omega)) &\textrm{ weakly *}.
	\end{array}\right.
\end{align*}
	
The function $v\in L^\infty(0,T; H^{s}_0(\Omega))\cap W^{1,\infty}(0,T, H^{-s}(\Omega))$ is a solution of the equation and, from the compactness of the embedding (see \cite{simon1986compact})
\begin{align*}
	L^\infty(0,T; H^{s}_0(\Omega))\cap W^{1,\infty}(0,T, H^{-s}(\Omega))\hookrightarrow C(0,T; H^{1-s}_0(\Omega))
\end{align*}
and \eqref{cu_norm} we deduce that $\norm{v_0}{H^{1-s}_0(\Omega)}=1$. In particular, thanks to the estimates \eqref{schr norm identities} we also have $\norm{v(x,t)}{H^{1-s}_0(\Omega)}\geq\PP_1$. On the other hand, \eqref{cu_int} implies $|v|/\delta^s=0$ on $\Sigma$ and, applying Lemma \ref{cu_lemma}, we immediately have $v\equiv 0$. This is in contradiction with the fact that $v$ has positive $H^{1-s}(\Omega)$-norm. Hence \eqref{cu_est} holds and the proof for $s>1/2$ is concluded. 
\end{proof}

\subsection{Proof of the observability inequality and controllability result}
This section is devoted to the proof of the observability inequality \eqref{schr obs}. In more detail, we are going to prove the following result.

\begin{theorem}\label{obs_thm}
Let $s\in[1/2,1)$ and let $\Omega$ and $\omega$ be as in the statement of Theorem \ref{schr_control_thm}. For any $v_0\in  L^2(\Omega)$, let $v=v(x,t)$ be the corresponding solution of \eqref{fr_schr_adj}.  
\begin{enumerate}
	\item[(i)] If $s\in(1/2,1)$, then for every $T>0$ there exists a positive constant $\C$, depending only on $s$, $T$, $N$, $\Omega$ and $\omega$, such that
	\begin{align}\label{schr_L2_control}
		\norm{v_0}{ L^2(\Omega)}^2\leq \C\int_0^T\norm{v(t)}{ L^2(\omega)}^2\,dt.
	\end{align}
	\item[(ii)] If $s=1/2$, then \eqref{schr_L2_control} holds for any $T>T_0$, where $T_0$ has been introduced in Proposition \ref{schr est thm}.
	\end{enumerate}
\end{theorem}

Before presenting the complete proof of Theorem \ref{obs_thm}, we introduce several preliminary technical lemmas. 

\begin{lemma}\label{lemma_Hs}
Let $s\in[1/2,1)$ and let $\Omega$ and $\omega$ be as in the statement of Theorem \ref{schr_control_thm}. For any $v_0\in \mathcal D(\mathcal A)$, let $v=v(x,t)$ be the corresponding solution of \eqref{fr_schr_adj}.  
\begin{enumerate}
	\item[(i)] If $s\in(1/2,1)$, then for every $T>0$ there exists a positive constant $\C$, depending only on $s$, $T$, $N$, $\Omega$ and $\omega$, such that
	\begin{align}\label{schr_grad_inequality}
		\norm{v_0}{ H^{s}_0(\Omega)}^2\leq \C\int_0^T\norm{v(t)}{H^{s}(\omega)}^2\,dt.
	\end{align}
	\item[(ii)] If $s=1/2$, then \eqref{schr_grad_inequality} holds for any $T>T_0$, where $T_0$ has been introduced in Proposition \ref{schr est thm}.
\end{enumerate}
\end{lemma}

\begin{proof}
First of all, we notice that in the statement of Lemma \ref{lemma_Hs}, we are distinguishing two cases: $s = 1/2$ and $s\in(1/2,1)$. The main difference between this two cases is the need of a strictly positive time $T_0>0$ for \eqref{schr_grad_inequality} to hold when $s=1/2$. This, in turn, is a consequence of the fact that in our proof we will employ  \eqref{schr est}.
	
Notwithstanding that, the procedure for proving \eqref{schr_grad_inequality} follows essentially the same path, both for $s>1/2$ and for $s=1/2$. Hence, in order to abridge our presentation, we are going to present here only the first case, $s>1/2$, leaving to the reader the proof for $s=1/2$. 
	
Thus, until the end of this proof, let us assume $s>1/2$. Moreover, throughout the proof, $\C$ will denote a generic positive constant independent of $v$. This constant may change even from line to line. 
	
Let us now recall the definition of the neighborhood of the boundary $\omega$ that we introduced in \eqref{neigh_bd_def}, which is 
\begin{align*}
	\omega:=\Omega\cap\mathcal{O}_{\varepsilon},\;\;\;\mathcal{O}_{\varepsilon}:=\bigcup_{x\in\Gamma_0} B(x,\varepsilon),
\end{align*} 
with $\Gamma_0$ as in \eqref{partition} (see Figure \ref{partition_fig}). Then, let us consider the cut-off function $\eta\in C^{\infty}(\RR^N)$ defined as follows
\begin{align}\label{eta}
	\begin{cases}
		\eta(x)\equiv 1, & x\in\widehat{\omega},
		\\ 
		0\leq\eta(x)\leq 1, & x\in\omega\setminus\widehat{\omega},
		\\ 
		\eta(x)\equiv 0, & x\in\Omega\setminus\omega,
	\end{cases}
\end{align}
where $\widehat{\omega}:=\Omega\cap\mathcal{O}_{\varepsilon_1}$, with $\varepsilon_1<\varepsilon$, is another neighborhood of the boundary, thinner than $\omega$ (see Figure \ref{neigh_bd}).
	
\begin{figure}[h]
	\includegraphics[scale=1]{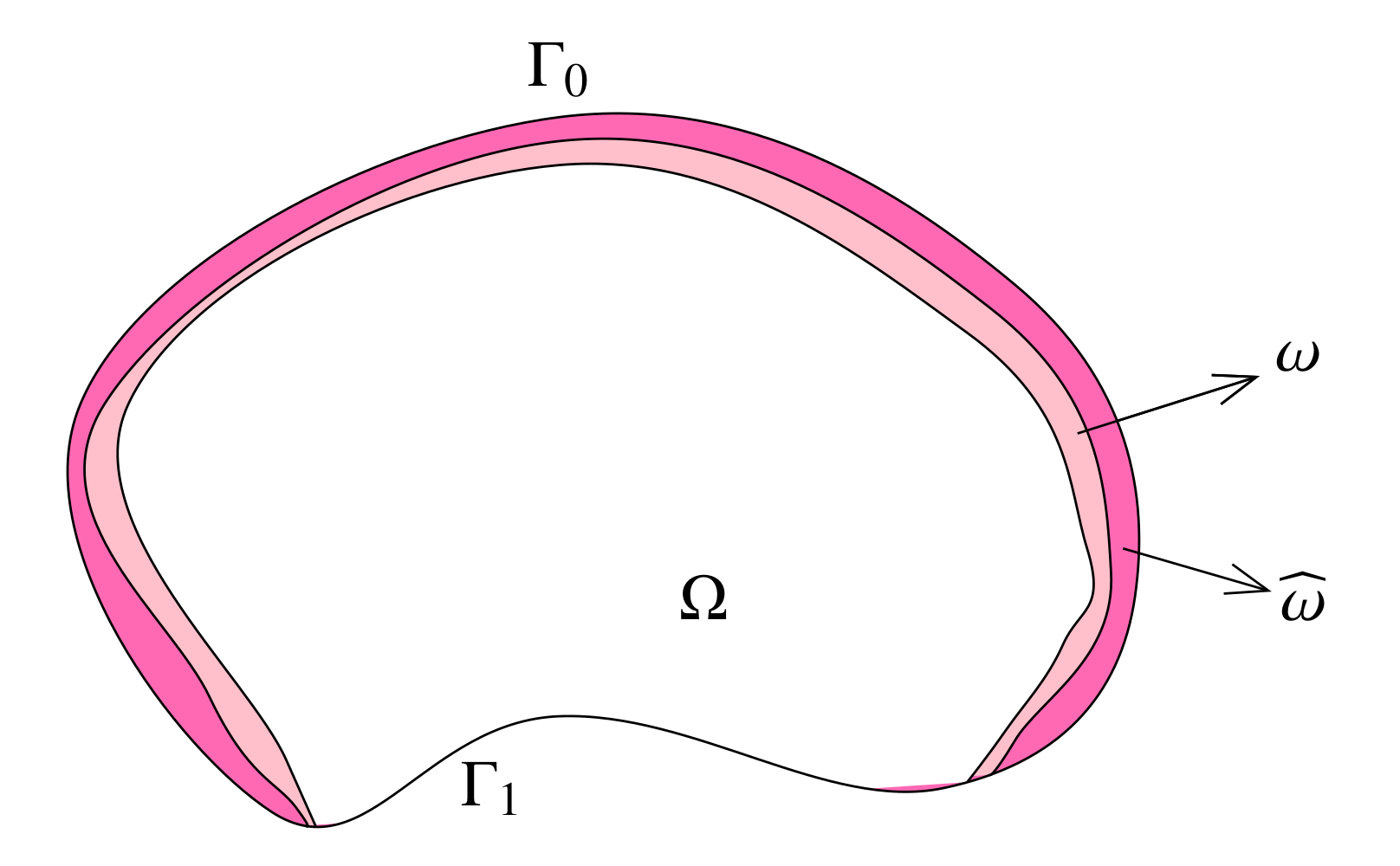}
	\caption{Example of the domain $\Omega$ with the partition of the boundary $(\Gamma_0,\Gamma_1)$ and the two neighborhood of the boundary $\widehat{\omega}$ and $\omega$.}\label{neigh_bd}
\end{figure}
	
Moreover, let us define $w(x,t):=\eta(x)v(x,t)$. It can be easily checked through the definition that the fractional Laplacian of $w$ is given by
\begin{align*}
	\fl{s}{w}=\fl{s}{(\eta v)}=\eta\fl{s}{v}+R
\end{align*}
where $R$ is the reminder term 
\begin{align}\label{R}
	R:= v\fl{s}{\eta} - \cns\,P.V.\int_{\RR^N} \frac{(v(x,t)-v(y,t))(\eta(x)-\eta(y))}{|x-y|^{N+2s}}\,dy.
\end{align}
Therefore, this new function $w$ satisfies the equation 
\begin{align*}
	\begin{cases}
		iw_t+\fl{s}{w}=R, & (x,t)\in Q
		\\ 
		w\equiv 0, & (x,t)\in Q^c
		\\ 
		w(x,0)=w_0, & x\in\Omega.
	\end{cases}
\end{align*}
	
In addition to that, following the procedure presented in \cite{biccari2017local}, employing the definition of $R$, the fact that the function $\eta$ is regular (in particular, Lipschitz), and the classical Cauchy-Schwarz and Young's inequalities, we can show that there exists a constant $\B>0$, not depending on $v$, such that 
\begin{align}\label{fl_norm_est_s}
	\norm{R(t)}{H^{1-s}(\omega)}\leq \B\norm{v(t)}{H^{s}(\omega)}.
\end{align}
Now, starting from \eqref{schr_pohozaev_nh} applied to $w$, we have
\begin{align*}
	\underbrace{\Gamma(1+s)^2\int_{\Sigma}\left(\frac{|w|}{\delta^s}\right)^2(x\cdot\nu)\,d\sigma dt}_{J} =&\; s\int_0^T\norm{\fl{\frac{s}{2}}{w(t)}}{ L^2(\omega)}^2\,dt+\Im \int_{\omega}\overline{w}(x\cdot\nabla w)\,\big |_0^T\,dx
	\\
	&- \Re\int_0^T\int_{\omega}R\Big(N\overline{w}+2x\cdot\nabla\overline{w}\Big)\,dxdt.
\end{align*}
Hence, applying the Cauchy-Schwarz inequality and \eqref{operator T estimate 2}, we obtain 
\begin{align}\label{J_estimate_1}
	J \leq \C\int_0^T\norm{w(t)}{ H^{s}(\omega)}^2\,dt + \C\int_0^T\norm{w(t)}{ L^2(\omega)}\norm{R(t)}{ L^2(\omega)}\,dt + \C\int_0^T\norm{w(t)}{H^{s}(\omega)}\norm{R(t)}{H^{1-s}(\omega)}\,dt.\notag
	\\
\end{align}
From \eqref{J_estimate_1}, using Young's inequality, and the fact that $\norm{\cdot}{L^2(\omega)}\leq\norm{\cdot}{H^{\beta}(\omega)}$ for any $\beta>0$, we get
\begin{align*}
	J \leq \C\int_0^T\norm{w(t)}{ H^{s}(\omega)}^2\,dt + \C\int_0^T\norm{R(t)}{H^{1-s}(\omega)}^2\,dt,
\end{align*}
from which it is straightforward to obtain 
\begin{align*}
	J \leq \C\int_0^T\norm{v(t)}{H^{s}(\omega)}^2\,dt + \C\int_0^T\norm{R(t)}{H^{1-s}(\omega)}^2\,dt.
\end{align*}
Finally, using \eqref{fl_norm_est_s}, we obtain the estimate
\begin{align*}
	J \leq \C\int_0^T\norm{v(t)}{ H^{s}(\omega)}^2\,dt.
\end{align*}
Hence, by using Theorem \eqref{schr est thm}, the inequality \eqref{schr_grad_inequality} follows.
\end{proof}

\begin{lemma}
Let $\Omega\subset\RR^N$ be a bounded regular domain, $s\in(0,1)$, $f\in  H^{-s}(\Omega)$, and let $v\in H_0^s(\Omega)$ be the solution of
\begin{align*}
	\begin{cases}
		\fl{s}{v}=f, & x\in\Omega,
		\\ 
		v\equiv 0, & x\in\Omega^c.
	\end{cases}
\end{align*}
Then, there exists a constant $\gamma>0$ such that 
\begin{align}\label{fr_duality_estimate}
	\norm{v}{H^{s}(\widehat{\omega})}^2\leq \gamma\left[\norm{f}{H^{-s}(\omega)}^2+\norm{v}{L^2(\omega)}^2\right].
\end{align}
\end{lemma}

\begin{proof}
Let us consider again the function $\eta(x)$ defined in \eqref{eta} and let $w=\eta v$. Thus, $w$ satisfies
\begin{align*}
	\begin{cases}
		\fl{s}{w}= \eta f + R:=g, & x\in\omega,
		\\
		w\in  H^{s}_0(\omega),
	\end{cases}
\end{align*}
where $R$ is the reminder term introduced in \eqref{R}.
	
We recall that we have $\norm{R}{ L^2(\omega)}<+\infty$. This, together with the assumption on $f$ and the definition of $\eta$, implies that $g\in  H^{-s}(\omega)$. Thus, there exists some positive constant $\gamma$, independent of $g$, such that (see \eqref{Hs_norm})
\begin{align*}
	\norm{w}{H^{s}(\omega)}^2\leq \gamma\norm{g}{H^{-s}(\omega)}^2.
\end{align*}
	
Expanding this last expression we easily obtain the existence of another positive constant, that we will still note by $\gamma$, such that 
\begin{align*}
	\norm{w}{H^{s}(\omega)}^2 &\leq \gamma \left[\norm{f}{H^{-s}(\omega)}^2 + \norm{v}{L^2(\omega)}^2\,\right].
\end{align*}
Hence, since 
\begin{align*}
	\norm{v}{H^{s}(\widehat{\omega})}^2 = \norm{w}{H^{s}(\widehat{\omega})}^2 \leq \norm{w}{H^{s}(\omega)}^2,
\end{align*}
we finally obtain the estimate \eqref{fr_duality_estimate}.
\end{proof}

\begin{lemma}
For any $T>0$ and $s\geq 1/2$, there exists a positive constant $\C$, depending only on $s$, $T$, $N$, $\Omega$ and $\omega$, such that for all $v$ solution of \eqref{fr_schr_adj} it holds
\begin{align}\label{schr_grad_dual_inequality}
	\norm{v_0}{H^{-s}(\Omega)}^2\leq \C\int_0^T\norm{v(t)}{ H^{-s}(\omega)}^2\,dt.
\end{align}
\end{lemma}

\begin{proof}
Let us define 
\begin{align*}
	\psi(x,t):=\int_0^t v(x,\tau)d\tau+\Theta(x),
\end{align*}
where
\begin{align*}
	\begin{cases}
		\fl{s}{\Theta}=-iv_0, & x\in\Omega,
		\\
		\Theta\in  H^{s}_0(\Omega).
	\end{cases}
\end{align*}
	
It can be readily checked that $\psi$ is a solution of \eqref{fr_schr_adj} with initial datum $\psi(x,0)=\Theta(x)$. Applying \eqref{schr_grad_inequality} to it we have 
\begin{align*}
	\norm{\Theta}{ H^{s}(\Omega)}^2\leq \C\int_0^T\norm{\psi(t)}{ H^{s}(\omega)}^2\,dt
\end{align*}
which, by elliptic regularity, and using \eqref{fr_duality_estimate}, becomes
\begin{align}\label{schr_H-s_inequality_3}
	\norm{v_0}{H^{-s}(\Omega)}^2\leq \C\int_0^T \left(\norm{\psi_t(t)}{ H^{-s}(\omega)}^2 + \norm{\psi(t)}{ L^2(\omega)}^2\,\right)\,dt.
\end{align}
From \eqref{schr_H-s_inequality_3}, proceeding by contradiction and using compactness, we finally obtain \eqref{schr_grad_dual_inequality}.
\end{proof}

\begin{proof}[Proof of Theorem \ref{obs_thm}]
	
First of all, recall that, from \eqref{schr_grad_inequality} and \eqref{schr_grad_dual_inequality}, we have
\begin{align}
	&\norm{v_0}{ H^{s}_0(\Omega)}^2\leq \C\int_0^T \norm{v(t)}{ H^{s}(\omega)}^2\,dt=\C\,\norm{v}{ L^2(0,T; H^{s}(\omega))}^2, \label{schr_L2_Hs_estimate}
	\\ 
	&\norm{v_0}{ H^{-s}(\Omega)}^2\leq \C\int_0^T \norm{v(t)}{ H^{-s}(\omega)}^2\,dt=\C\,\norm{v}{ L^2(
		0,T; H^{-s}(\omega))}^2. \label{schr_L2_H-s_estimate}
\end{align}
We are going to prove \eqref{schr_L2_control} by interpolation. To this end, let us consider the linear operator 
\begin{align*}
	\Lambda: H^{-s}(\Omega)\rightarrow  L^2(0,T; H^{-s}(\omega))
\end{align*} 
defined by 
\begin{align*}
	\Lambda v_0:=\left.\left(e^{it\fl{s}{}}v\right)\right |_{\omega}.
\end{align*}
Clearly, 
\begin{align*}
	\norm{\Lambda v_0}{ L^2(0,T; H^{-s}(\omega))}\leq \C\,\norm{v_0}{ H^{-s}(\Omega)}.
\end{align*}
Furthermore, from \eqref{schr_L2_H-s_estimate} it follows that 
\begin{align*}
	\norm{\Lambda v_0}{ L^2(0,T; H^{-s}(\omega))}\geq \C\,\norm{v_0}{ H^{-s}(\Omega)}.
\end{align*}
	
Therefore, we can consider the closed subspace $X_0:=\Lambda( H^{-s}(\Omega))$ of $ L^2(0,T; H^{-s}(\omega))$ and the linear operator $\Pi:=\Lambda^{-1}$ (since $\Lambda$ is an isomorphism between $ H^{-s}(\Omega)$ and $X_0$). Thus, 
\begin{align}\label{operator_Pi_1}
	\Pi\in\mathcal{L}(X_0,Y_0),
\end{align} 
with $Y_0:= H^{-s}(\Omega)$. If now we set $X_1:=X_0\cap  L^2(0,T; H^{s}(\omega))$, it follows from \eqref{schr_L2_Hs_estimate} that
\begin{align}\label{operator_Pi_2}
	\Pi\in\mathcal{L}(X_1,Y_1),
\end{align} 
with $Y_1:= H^{s}(\Omega)$. From \eqref{operator_Pi_1}, \eqref{operator_Pi_2} and \cite[Theorem 5.1]{lions1968problemes}, we have $\Pi\in\mathcal{L}([X_0,X_1]_{1/2},[Y_0,Y_1]_{1/2})$. Moreover, \cite[Lemma 12.1]{lions1968problemes} yields $[Y_0,Y_1]_{1/2}= L^2(\Omega)$ and from \cite[Theorem 5.1.2]{bergh1976interpolation} we conclude that
\begin{align*}
	[ L^2(0,T; H^{s}(\omega)), L^2(0,T; H^{-s}(\omega))]_{1/2}= L^2(0,T;[ H^{s}(\omega); H^{-s}(\omega)]_{1/2})= L^2(0,T; L^2(\omega)).
\end{align*}
	
Hence, since $X_0$ and $X_1$ are closed subspaces of $ L^2(0,T; H^{-s}(\omega))$ and $ L^2(0,T; H^{s}(\omega))$ respectively, using \cite[Theorem 15.1]{lions1968problemes} we can verify that the norm of the space $[X_0,X_1]_{1/2}$ is equivalent to the norm of $ L^2(0,T; L^2(\omega))$ and, since $\Pi\in\mathcal{L}([X_0,X_1]_{1/2}; L^2(\Omega))$, we finally obtain \ref{schr_L2_control}.
\end{proof}

Having proved the observability of \eqref{fr_schr_adj} from a neighborhood of the boundary of the domain, our controllability theorem is now a direct consequence of a duality argument. This argument being classical (see, e.g., \cite{coron2009control,lions1988controlabilite,lions1988exact}) we are going to omit it here.

\section{Conclusion and open problems and perspectives}\label{open_pb_sec}

In this paper, we have analyzed the interior controllability properties of a non-local Schr\"odinger-type equation involving the fractional Laplace operator $\fl{s}{}$ on a bounded $C^{1,1}$ domain. 

We firstly considered the one-dimensional case, and we employed spectral techniques to prove that the fractional Schr\"odinger equation is null controllable for $s\geq 1/2$, while for $s<1/2$ null controllability fails. Besides, when $s=1/2$ we showed that the null-controllability result holds true provided the time horizon $T$ is large enough.

These one-dimensional results have been then extended to the multi-dimensional case via the employment of multiplier techniques combined with the Pohozaev identity for the fractional Laplacian. 

We conclude this work by briefly presenting some suggestions for open problems.

\begin{itemize}
	\item[•] \textit{Exterior controllability for the fractional Schr\"odinger equation.\,}
	The concept of exterior controllability for evolution equations involving the fractional Laplacian has been recently introduced in several contributions (see \cite{antil2019controllability,louis2018approximate,warma2019approximate,warma2018null,warma2018analysis}). This is the equivalent of the boundary controllability property for local partial differential equations. It takes into account the non-locality of the fractional Laplacian, which yields the ill-posedness of boundary value problems associated to this operator. To the best of our knowledge, so far the exterior controllability property has been analyzed only in the context of fractional heat and wave equations. To consider the case of a fractional Schr\"odinger equation is then an interesting open problem.
	\item[•] \textit{Micro-local analysis for the solutions of evolution equations with the fractional Laplacian.\,} Geometric Optics expansion for the solutions of an evolution PDE is a very powerful tool that, if well developed, can provide relevant information on the propagation of a wave-type equation and on the way in which its solutions interact with the boundaries or interfaces in the domain of definition (see, e.g. \cite{ralston1982gaussian,rauch2005polynomial}). This tool has also been applied, for instance in \cite{bardos1992sharp}, to the study of controllability properties for different classes of partial differential equations. As far as we know, a complete analysis of micro-local properties of non-local models such as the ones we considered in this contribution has not been developed yet and it would then be an interesting research direction. 
	\item[•] \textit{Rigorous justification of the minimal controllability time in the case $s=1/2$.\,} In the proof of Theorem \ref{schr est thm}, when $s=1/2$, we introduced a minimal (strictly positive) time for \eqref{schr est} to hold. This came from the employment of \eqref{est_mintime} and from the fact that, if $s=1/2$, the spaces $H^s_0(\Omega)$ and $H^{1-s}_0(\Omega)$ clearly coincide and no compactness argument may be used. In particular, in this case, if $T$ is not large enough the constant in the lower estimate in \eqref{est_mintime_constant} would be negative. Although this minimal time requirement may appear a technicality associated with our strategy for proving Theorem \ref{schr est thm}, we actually believe that this constraint cannot be removed. Indeed, through a simple stationary phase approach (see \cite[Chapter 3]{whitham1999linear}), it is possible to show that, for $s=1/2$, the velocity of propagation of the solution of \eqref{fr_schr_control} is finite (whereas it is infinite when $s>1/2$). By means of standard micro-local analysis techniques (see, e.g. \cite{bardos1992sharp}) one then expects that it is required a minimal positive time $T_0>0$ for the observability inequality \eqref{schr obs} (hence, for the controllability of \eqref{fr_schr_control}) to hold. A rigorous justification of this simple observation is, to the best of our knowledge, still missing and represent a very interesting issue which deserves a deeper investigation.
\end{itemize}

\bibliographystyle{ams}
\bibliography{biblio}

\end{document}